\documentclass{article}
\usepackage{amssymb,amsfonts,amsmath,amsthm,amscd,dsfont,mathrsfs,mathtools,nicefrac}
  \mathtoolsset{showonlyrefs}

\usepackage{caption}
\usepackage{authblk}
\usepackage{etoc}
\usepackage{wrapfig}

\usepackage{enumerate}
\usepackage{xcolor}
\usepackage{graphicx}
\usepackage{fullpage}
\usepackage{hyperref}

\hypersetup{
    colorlinks,
    linkcolor={blue!80!black},
    citecolor={green!50!black},
}
\colorlet{linkequation}{blue}

\usepackage{caption}
\usepackage{authblk}
\usepackage{etoc}
\usepackage{wrapfig}

\newcommand{\tr}{\ensuremath{{\scriptscriptstyle\mathsf{\,T}}}}

\newcommand{\norm}[1]{\left\|#1\right\|}

\newcommand{\by}{\boldsymbol{y}}

\newcommand{\bw}{\boldsymbol{w}}

\def\ddefloop#1{\ifx\ddefloop#1\else\ddef{#1}\expandafter\ddefloop\fi}
\def\ddef#1{\expandafter\def\csname c#1\endcsname{\ensuremath{\mathcal{#1}}}}

\ddefloop ABCDEFGHIJKLMNOPQRSTUVWXYZ\ddefloop
\def\ddef#1{\expandafter\def\csname s#1\endcsname{\ensuremath{\mathsf{#1}}}}
\ddefloop ABCDEFGHIJKLMNOPQRSTUVWXYZ\ddefloop

\ddefloop ABCDEFGHIJKLMNOPQRSTUVWXYZ\ddefloop
\def\ddef#1{\expandafter\def\csname b#1\endcsname{\ensuremath{\mathbf{#1}}}}
\ddefloop ABCDEFGHIJKLMNOPQRSTUVWXYZ\ddefloop

\usepackage[mathscr]{euscript}
\usepackage{accents}

\DeclareSymbolFont{rsfs}{U}{rsfs}{m}{n}
\DeclareSymbolFontAlphabet{\mathscrsfs}{rsfs}

\renewcommand{\P}{\mathbb{P}}

\newcommand{\R}{\mathbb{R}}
\newcommand{\E}{\mathbb{E}}
\newcommand{\Var}{\text{Var}}

\newcommand{\N}{\mathbb{N}}

\newcommand{\eps}{\varepsilon}

\def\bA{{\boldsymbol A}}

\def\bD{{\boldsymbol D}}

\def\bH{{\boldsymbol H}}

\def\bK{{\boldsymbol K}}

\def\bW{{\boldsymbol W}}
\def\bX{{\boldsymbol X}}

\def\bu{{\boldsymbol u}}
\def\bv{{\boldsymbol v}}
\def\bw{{\boldsymbol w}}
\def\bx{{\boldsymbol x}}
\def\by{{\boldsymbol y}}
\def\bz{{\boldsymbol z}}

\def\bbeta{{\boldsymbol \beta}}

\def\beps{{\boldsymbol \eps}}

\def\bphi{{\boldsymbol \phi}}

\def\bPsi{{\boldsymbol \Psi}}
\def\bPhi{{\boldsymbol \Phi}}

\def\RF{{\rm RF}}

\def\cP{{\mathcal P}}

\def\cP{{\mathcal P}}

\def\RF{{\sf RF}}

\def\RF{{\sf RF}}

\usepackage{color}

\usepackage{enumerate}

\newtheorem{theorem}{Theorem}[section]
\newtheorem{lemma}[theorem]{Lemma}
\theoremstyle{definition}
\newtheorem{definition}[theorem]{Definition}

\newtheorem{prop}[theorem]{Proposition}

\newtheorem{cor}[theorem]{Corollary}
\newtheorem{assumption}[theorem]{Assumption}

\theoremstyle{remark}
\newtheorem{remark}[theorem]{Remark}
\numberwithin{equation}{section}
\newcommand{\Id}{\operatorname{Id}}

\newcommand{\train}{\text{train}}

\newcommand{\Tr}{\operatorname{Tr}}
\renewcommand{\tr}{\operatorname{tr}}
\newcommand{\x}{\boldsymbol{x}}

\newcommand{\z}{\boldsymbol{z}}

\newcommand{\w}{\mbi{w}}

\renewcommand{\P}{\mathbb{P}}

\newcommand{\KNl}{\bK_{N,\lambda}}
\newcommand{\tKNl}{\tilde{\bK}_{N,\lambda}}
\newcommand{\Kl}{\bK_{\lambda}}
\newcommand{\tKl}{\tilde{\bK}_{\lambda}}
\newcommand{\KmN}{\bK_{m,N}}
\newcommand{\Km}{\bK_m}
\newcommand{\KxxN}{\bK^{(2)}_N}
\newcommand{\Kxx}{\bK^{(2)}}
\newcommand{\GCV}{\text{GCV}}
\newcommand{\CV}{\text{CV}}

\def\mbi#1{\boldsymbol{#1}}
\def\v#1{\mbi{#1}} 
\def\mbi#1{\boldsymbol{#1}} 

\def\mbiW{\mbi{W}}
\def\mbiX{\mbi{X}}

\def\KK{{\sf K}}
\def\Parg#1{\P\left({#1}\right)}

\begin{document}

\title{Overparameterized Random Feature Regression with Nearly Orthogonal Data}

\author{Zhichao Wang\thanks{University of California, San Diego.  \texttt{zhw036@ucsd.edu}.}  \quad and\quad 
Yizhe Zhu\thanks{ University of California, Irvine.  \texttt{yizhe.zhu@uci.edu.}}} 

\maketitle
\begin{abstract}
We investigate the properties of random feature ridge regression (RFRR) given by a two-layer neural network with random Gaussian initialization. We study the non-asymptotic behaviors of the RFRR with nearly orthogonal deterministic unit-length input data vectors in the overparameterized regime, where the width of the first layer is much larger than the sample size. Our analysis shows high-probability non-asymptotic concentration results for the training errors, cross-validations, and generalization errors of RFRR centered around their respective values for a kernel ridge regression (KRR). This KRR is derived from an expected kernel generated by a nonlinear random feature map. We then approximate the performance of the KRR by a polynomial kernel matrix obtained from the Hermite polynomial expansion of the activation function, whose degree only depends on the orthogonality among different data points. This polynomial kernel determines the asymptotic behavior of the RFRR and the KRR. Our results hold for a wide variety of activation functions and input data sets that exhibit nearly orthogonal properties. Based on these approximations, we obtain a lower bound for the generalization error of the RFRR for a nonlinear student-teacher model.
\end{abstract}

\section{Introduction}
Random feature  regression is closely linked  to deep learning theory as a linear model with respect to random  features. Training the output layer weight with ridge regression for a neural network with \textit{random} first-layer weight is equivalent to a random feature ridge regression model (RFRR) \cite{rahimi2007random,cho2009kernel,daniely2016toward,poole2016exponential,schoenholz2017deep,lee2018deep,matthews2018gaussian}. The conjugate kernel (CK), whose spectrum has been exploited to study the generalization of random feature regression \cite{mei2021generalization}, is the Gram matrix of the output of the last hidden layer on the training dataset. The performances (e.g., prediction risk) have been studied by \cite{rahimi2007random,rahimi2008weighted,rudi2016generalization,mei2019generalization,mei2021generalization,ghorbani2021linearized}. As the width of the neural network increases to infinity, we expect the empirical CK concentrates around its expectation, analogously to the neural tangent kernel (NTK) theory from \cite{jacot2018neural}. In this overparameterized (or ultra-wide \cite{arora2019exact}) regime, RFRR is asymptotically equivalent to a kernel ridge regression (KRR) model.

In this paper, we focus on the random CK generated by a two-layer fully-connected neural network at random initialization $f: \mathbb R^{d\times n} \to \mathbb R^{n}$ such that
\begin{equation}\label{eq:1hlNN}
    f(\mbiX):=\frac{1}{\sqrt{N}}\v\theta^\top \sigma\left(\mbiW \mbiX\right),
\end{equation}
where $\mbiX\in\R^{d\times n}$ is the input data matrix, $\mbiW\in\R^{N\times d}$ is the weight matrix for the first layer, $\v\theta\in\R^{N}$ is the second layer weight, and $\sigma$ is a nonlinear activation function. Here $d$ is the feature dimension, $n$ is the sample size of the input data, and $N$ is the width of the first layer. 

This work focuses on the behavior of the two-layer network under the random initialization of $\mbiW$ with sufficiently large width $N$. We will always view the input data $\mbiX$ as a deterministic matrix (independent of the random weights in $\mbiW$) with certain assumptions. We fix the random matrix $\bW$ and only train the second layer $\v\theta$ via training data $\bX$. This procedure is the same as the linear regression of random feature vectors $\{\sigma(\bW\bx_i)\in\R^N$, $i\in [n]\}$. 
The empirical CK matrix is defined by 
\begin{equation}\label{eq:K_N(X)}
\bK_N:=\frac{1}{N}\sigma\left(\mbiW \mbiX\right)^\top \sigma\left(\mbiW \mbiX\right)\in\R^{n\times n}.
\end{equation}
We will show that this random CK matrix will be concentrated around its expected $n\times n$ kernel matrix
\begin{equation}\label{def:phi}
    \bK:=\E \bK_N=\E_{\w}[\sigma(\w^\top \bX)^\top \sigma(\w^\top \bX)],
\end{equation}
under the spectral norm when width $N$ is sufficiently large,
where $\w$ is the standard normal random vector in $\R^{d}$. 
Random feature regression has already attracted as a random approximation of the reproducing kernel Hilbert space (RKHS) defined by population kernel function $\bK: \R^{d} \times \R^{d} \to \R$ such that
\begin{equation}\label{eq:def_K}
    \bK(\x_1,\x_2):=\E_{\w} [\sigma (\langle \w,\x_1 \rangle ) \sigma (\langle \w,\x_2 \rangle )],
\end{equation}
when width $N$ is sufficiently large \cite{rahimi2007random,bach2013sharp,rudi2016generalization,bach2017equivalence,mei2021generalization}.  By an abuse of notation, we use $\bK$ to represent both the $n\times n$ kernel matrix $\bK(\bX,\bX)$ depending on dataset $\bX$ and the kernel function in \eqref{eq:def_K}. Denote the output of the first layer by 
\begin{equation}
    \bPhi:=\sigma(\bW\bX) \in \mathbb R^{N\times n}. \label{eq:Y}
\end{equation}
Observe that the rows of the matrix $\bPhi$ are independent and identically distributed since only $\bW$ is random and $\bX$ is deterministic. Let the $i$-th row of $\bPhi$ be $\bphi_i^\top=\sigma(\bw_i\bX)$ for $1\le i\le N$, where we denote $\bw_i\in\R^d$ as the $i$-th row of weight $\bW$. Then, CK can be written as $\bK_N=\frac{1}{N}\sum_{i=1}^{N}\bphi_i\bphi_i^\top,$
which is a sum of $N$ independent rank-one random matrices in $\R^{n\times n}$. The second moment of any row $\bphi_i$ is given by \eqref{def:phi}.

Most of the recent results considered the RFRR with the data points $\bX$ independently drawn from a specific high-dimensional distribution, e.g.,  uniform measure on the hypercube or the unit sphere \cite{misiakiewicz2022spectrum,hu2022sharp,xiao2022precise,ghorbani2021linearized} or under the hypercontractivity  assumption from \cite{mei2021generalization}. The analysis of this RFRR usually requires strong assumptions on the data distribution and specific orthogonal polynomial expansions with respect to the distribution. In practice, real-world data cannot satisfy these ideal assumptions, or it is hard to verify them. In this paper, we consider a general deterministic dataset for RFRR. Inspired by \cite{du2019gradienta,fan2020spectra,wang2021deformed,donhauser2021rotational}, we point out that the inner products among different unit-length data points, namely the degree of the orthogonality, play an important role in the performances of the RFRR. More precisely, it affects how many degrees of the polynomial this RFRR can consistently learn from the teacher models. The expected kernel model can be truncated as a polynomial inner-product kernel based on this approximate orthogonality of the data points. Combining the concentration of RFRR and this polynomial truncation, we can obtain a lower bound of the generalization error (out-of-sample prediction risk) for RFRR induced by an ultra-wide neural network ($N\gg n$).  
Since we consider a general distribution-free dataset, we can also analyze cross-validations of RFRR approximated by corresponding cross-validations of the KRR. Our assumptions on the dataset are verifiable even for real-world datasets, and our theory exhibits new ingredients to the  study of neural networks with general real-world datasets \cite{liao2018spectrum,goldt2020gaussian,wei2022more}.

\subsection{Our Contributions}
We prove a sequence of sharp concentrations for RFRR around its expected KRR for a general \textit{distribution-free} dataset satisfying an $\ell$-orthonormal property (see Assumption~\ref{assump:asymptotics}). As long as the width $N$ of the neural network is much larger than sample size $n$, we can use a KRR to approximate RFRR in terms of in-sample prediction risks, cross-validations, and out-of-sample prediction risks. With a qualitative control of the approximate orthogonality  among different data points measured by $ \left\| (\bX^\top \bX)^{\odot (\ell+1)}-\Id\right\|_F$, we can further approximate this KRR by a truncated polynomial inner-product KRR. Meanwhile, we reveal that both RFRR and its corresponding KRR can only consistently learn a polynomial teacher model with a degree at most $\ell$. To the best of our knowledge, this is the first work making a connection between the lower bound of the generalization errors of RFRR and KRR, and the orthogonality of deterministic data points. Our main results are stated in Section~\ref{sec:mainresults} and proved in Appendix~\ref{sec:proofs}. The empirical simulations on both synthetic and real-world datasets are presented in Section~\ref{sec:simulate}.

\subsection{Related Work}
\paragraph{Nonlinear Random Matrix Theory}
When $N\asymp n $, the concentration of the CK matrix around its expectation fails, and the limiting spectrum of the CK with random input dataset has been investigated by~\cite{pennington2017nonlinear,benigni2019eigenvalue,louart2018random,benigni2022largest}; whereas~\cite{fan2020spectra} studied the spectrum of the CK with similar but stronger assumptions compared to ours on input data and activation functions, and obtained  a deformed Marchenko-Pastur distribution~\cite{fan2020spectra}.  As an application, when $N\asymp n$, the behavior of RFRR is  determined by the limiting spectra of the CK ~\cite{gerace2020generalisation,mei2019generalization,adlam2020neural,chouard2022quantitative}. Specifically, \cite{louart2018random,liao2020ARM,hu2020universality,chouard2022quantitative} studied the training error and empirical test error of RFRR in the proportional limit.
\paragraph{Concentrations of RFRR}
\cite{rudi2016generalization} proved the approximation of RFRR when the sample size $n$ and the number of neurons (width) $N$ satisfy $N\asymp \sqrt{n}\log n$. This condition only considered fixed $d$ with i.i.d. data.  Moreover, \cite{louart2018random,wang2021deformed} considered similar concentrations of RFRR for more general datasets. The concentration of random Fourier feature matrices was considered by \cite{Chen2022}. The sharp analysis of RFRR \cite[Theorem 1]{mei2021generalization} gave the precise asymptotic behavior of RFRR and only required $N\gg n$. Our results are consistent with their results on the training errors but relax the assumption on the dataset.
\paragraph{Rotational Invariant Kernels}
The expected CK and NTK are rotational invariant kernels \cite{liang2020multiple}, whence the kernel theory plays a crucial role in analyzing ultra-wide neural networks. In general, the spectra of rotational invariant kernels have been analyzed by \cite{el2010spectrum,liao2019inner,ali2021random} when $n\asymp d$ and such results have been applied in the study of kernel ridge regression in \cite{bartlett2021deep,sahraee2022kernel}. \cite{liao2018spectrum,liao2019inner} studied the inner-product kernel induced by random features in the proportional limit, where they can further decompose the expected kernel and extract the
useful structure from the data. When $n\asymp d^k$, for $k\in\N$, the performance of inner-product kernel with data uniformly drawn from the unit sphere has been recently studied by \cite{misiakiewicz2022spectrum,hu2022sharp,lu2022equivalence,xiao2022precise}.  
\paragraph{Cross-validations in High Dimensions}
There is a line of research on cross-validations \cite{liu2019ridge,jacot2020kernel,miolane2021distribution,xu2021consistent,hastie@2022surprises,meanti2022efficient} for ridge regressions. In high dimensional linear ridge regressions, \cite{hastie@2022surprises} shows precise asymptotic behaviors of cross-validations as $n/d\to \gamma\in(0,\infty)$. Cross-validations help us to tune the hyperparameters and approximate the generalization error of the model \cite{jacot2020kernel,wei2022more}. Most of the above works only focus on linear regression, while our work considers the cross-validations of both nonlinear RFRR and KRR on general datasets.
 
\section{Main results}\label{sec:mainresults}

\paragraph{Notations}We use $\tr(A)=\frac{1}{n}\sum_{i} A_{ii}$ as the normalized trace of a matrix $A\in\R^{n\times n}$ and $\Tr(A)=\sum_{i} A_{ii}$. Denote vectors by lowercase boldface. $\| A \|$ is the spectral norm for any matrix $A$,  $\|A\|_F$ denotes the Frobenius norm, and $\|\x\|$ is the $\ell_2$-norm of any vector $\x$. Denote $A\odot B$ as the Hadamard product of two matrices $A,B$ of the same size defined by $(A\odot B)_{ij}=A_{ij}B_{ij}$, and $A^{\odot k}$ is the $k$-th Hadamard product of $A$ with itself.
Let $\E_{\w}[\cdot]$ be the expectation with respect to the random vector $\w$.

\subsection{Model Assumptions}

Before stating our main results, we list the following assumptions for the random weights $\bW$,  the activation function $\sigma$, and input data $\bX$. 
\begin{assumption}\label{assump:W}
The entries of weight matrix $\bW\in\R^{N\times d}$ are i.i.d.\  standard normal random variables $\cN(0,1)$.
\end{assumption}

Let $h_k$ be the $k$-th normalized Hermite polynomial and $\zeta_k(\sigma)$ be the $k$-th Hermite coefficient for nonlinear function $\sigma$. For more details, see Definition~\ref{eq:hermitepolynomial}.

\begin{assumption}\label{assump:sigma} We assume $\sigma$ has a polynomial growth rate: $|\sigma(x)|\leq C_{\sigma} (1+|x|)^{C_{\sigma}}$ for a constant $C_{\sigma}\geq 0$.
Denote the standard Gaussian measure by $\Gamma$. Define  the $L^2(\Gamma)$ and $L^4(\Gamma)$ norms of $\sigma$ by
$
   \|\sigma\|_2= (\E[\sigma^2(\xi)])^{1/2}$ and  $\|\sigma\|_4= (\E[\sigma^4(\xi)])^{1/4}
   $ respectively,
where  $\xi\sim \cN(0,1)$. 

\end{assumption} 
In particular, Assumption \ref{assump:sigma} is similar to \cite{montanari2020interpolation}, and it covers many commonly used activation functions, including sigmoid, tanh, ReLU, and leaky ReLU. 
This is a more general condition compared to previous works by \cite{wang2021deformed}, which assume that $\sigma$ is  Lipschitz, and Assumption \ref{assump:sigma} is sufficient for the concentrations of training and generalization errors for RFRR. 

We  consider a sequence of $\bX_n \in \mathbb R^{d_n\times n}$ with growing $d_n$ as $n\to\infty$, where all $\bX_n$ satisfy the following assumption. Below we drop the dependence on $n$ for the ease of notations. We treat  $\bX$ as a deterministic matrix under the following asymptotic condition.

\begin{assumption}[$\ell$-orthonormal dataset]\label{assump:asymptotics}
 Suppose that the input data $\bX\in\R^{d\times n}$ satisfies $\|\x_i\|=1,\forall i\in [n]$. Let $\ell\in \mathbb N$ be the smallest integer  such that
\begin{align}
    \lim_{n\to\infty}\left\| (\bX^\top \bX)^{\odot (\ell+1)}-\Id\right\|_F =0.\label{assump:asymptotic_l}
\end{align}
 We further assume $\sigma_{>\ell}^2:=\|\sigma\|_2^2-\sum_{k=1}^\ell\zeta_k^2(\sigma)>0$.
\end{assumption}

Different from previous work that requires an upper bound on the maximal angle $\eps_n:=\max_{i\not=j} |\langle \bx_i,\bx_j\rangle |$    \cite{fan2020spectra,wang2021deformed,nguyen2020global,hu2020surprising,frei2022implicit}, our relaxed Condition \eqref{assump:asymptotic_l}  measures how data points separate from each other \textit{on average}. In particular, 
\begin{equation}\label{assump:asymptotic_l_eps}
    \left\| (\bX^\top \bX)^{\odot (\ell+1)}-\Id\right\|_F\leq n\eps_n^{\ell+1},
\end{equation}
whence \eqref{assump:asymptotic_l} holds if $n\eps_n^{\ell+1}\to 0$.
Here, feature dimension $d$ of the data is implicitly governed by \eqref{assump:asymptotic_l}. In a word, degree $\ell$ in \eqref{assump:asymptotic_l_eps} exhibits the average degree of the orthogonality among different data points.

We can also verify Assumption \ref{assump:asymptotics} for a random dataset.
For example, if $\{\bx_i\}_{i\in [n]}$ are i.i.d.  uniformly distributed on $\mathbb S^{d-1}$ and $n=\Theta(d^{\alpha})$ for $\alpha \in \mathbb R_+$, then $\eps_n=O\left( \frac{\log^{1/2} n}{d^{1/2}}\right)$ with high probability (see, for example, \cite{vershynin2018high}),  and we can take $\ell=2\lfloor\alpha \rfloor$ and condition on the high probability event to make $\bX$ deterministic. A similar argument is also applied by \cite{donhauser2021rotational}, where the distribution of random data can have some covariance structure.

\subsection{Power Expansion of the Expected Kernel}\label{sec:approx_kernel}

For any two unit-length column vectors $\bx_\alpha, \bx_\beta$ in $\bX$, and any two Hermite polynomials $h_j,h_k$, we have \cite[Lemma D.2]{nguyen2020global}
\begin{align}\label{eq:orthogonal_relation}
    \E_{\w}[h_j(\langle \w,\x_{\alpha}\rangle)h_k(\langle \w,\x_{\beta}\rangle)  ]= \delta_{jk} \langle \x_{\alpha}, \x_{\beta} \rangle ^k.
\end{align}
This relation also appears in \cite{oymak2020toward}, which directly gives the following power expansion of the expected kernel $\bK$  in \eqref{def:phi}: $\bK=\sum_{k=0}^{\infty} \zeta_k^2(\sigma)\left(\bX^\top \bX\right)^{\odot k}$.
 Hence, the kernel function  $\bK$ defined in \eqref{eq:def_K} is an inner-product kernel. In a concurrent work by \cite{murray2022characterizing}, the same power series expansion was applied to the NTK.

In high-dimensional statistics, invariant kernels can be approximated by some simpler models. For instance, \cite{el2010spectrum} proved that the inner-product random kernel matrices with a random dataset could be approximated by a linear random matrix model when $d\asymp n$. The proof by \cite{el2010spectrum}  utilized the Taylor approximation of the nonlinear function. In this work, beyond the first-order approximation in \cite{el2010spectrum}, we define a degree-$\ell$ polynomial inner-product kernel by
 \begin{align}
     \bK_\ell:=\sum_{k=0}^\ell \zeta_k^2(\sigma)\left(\bX^\top \bX\right)^{\odot k}+\sigma_{>\ell}^2\Id,\label{def:Phi0}
 \end{align}
Here $\sigma_{>\ell}^2$ is an extra ridge parameter added to the polynomial kernel $\sum_{k=0}^\ell \zeta_k^2(\sigma)\left(\bX^\top \bX\right)^{\odot k}$. This extra ridge can be viewed as an \textit{implicit regularization}, especially for the minimum-norm interpolators \cite{liang2020multiple,jacot2020implicit,bartlett2021deep}. 
 
Assumption~\ref{assump:asymptotics} implies that the off-diagonal entries of $\left(\bX^\top \bX\right)^{\odot k}$ become negligible when the power $k$ is sufficiently large. Hence, we can truncate $\bK$ and employ $\bK_\ell$ as an approximation of $\bK$ as follows.
\begin{prop}\label{prop:pkrr}
Under Assumptions \ref{assump:sigma} and \ref{assump:asymptotics}, let $n_0$ be the smallest integer such that for all $n\geq n_0$, $\max_{i\not=j}\big|\x_{i}^\top \x_{j}\big| \leq 1/\sqrt{2}$ and
\begin{align} 
    \left\| (\bX^\top \bX)^{\odot (\ell+1)}-\Id\right\|_F & \leq  \frac{\sigma_{>\ell}^2}{4\|\sigma\|_4^2}. \label{assump:relaxed_l} 
\end{align}
We have for all $n\geq n_0$, $\lambda_{\min}( \bK) \geq \lambda_0:=\frac{1}{2}\sigma_{>\ell}^2$, and
\begin{align}
    \|\bK_\ell-\bK\| &~\le \sqrt{2} \|\sigma\|_4^2 \left\| \left(\bX^\top \bX\right)^{\odot \ell+1}-\Id\right\|_F\leq  \frac{\sigma_{>\ell}^2}{2\sqrt{2}} .\label{eq:F_norm_bound}
\end{align}
\end{prop}

\begin{remark}[Comparison to previous work with random dataset]\label{rmk:compare_to_random}
    \eqref{eq:F_norm_bound} is proved by using the inequality $\|\bK_\ell-\bK\|\leq \|\bK_\ell-\bK\|_F$ and performing an entry-wise expansion of $(\bK_\ell-\bK)$. Such a Hermite polynomial expansion approach might not be optimal if we know the exact distribution of the random dataset.  Previous work from \cite{ghorbani2021linearized,mei2021generalization,montanari2020interpolation,hu2022sharp} assumed random datasets and random weights with specific distributions. The authors obtained better approximation error bounds  using a harmonic analysis approach, where the activation functions and the kernel $\bK$  were expanded in terms of an orthogonal basis with respect to the  distribution of random $\bX$ and $\bW$.  In many examples, these two distributions are assumed to be the same, which provides a convenient way to expand  and approximate $\bK$ with some degree-$\ell$ polynomial kernel. 
    Since we do not have any specific data distribution assumption, such an approach cannot be applied to  deterministic datasets. 
    
\begin{remark}[Optimality]
       In fact, under our Assumption \ref{assump:asymptotics}, the bound \eqref{eq:F_norm_bound} is tight up to a constant factor. For example, let $\sigma(x)=\sum_{k=0}^{\ell+1}\zeta_k(\sigma)h_k(x)$ be an order-$(\ell+1)$ polynomial with $\zeta_{\ell+1}(\sigma)\not=0$. Assume $|\langle \bx_i,\bx_j\rangle| =\eps$ for all $i\not=j$ and $\ell$ is an odd integer. Then 
    \begin{align}
         \|\bK_{\ell}-\bK\| =~&\xi_{\ell+1}^2(\sigma) \left\|  \left(\bX^\top \bX\right)^{\odot (\ell+1)}-\Id \right\|
         \geq ~\xi_{\ell+1}^2(\sigma)\eps^{\ell+1} (n-1)
         \geq ~\frac{1}{2} \xi_{\ell+1}^2(\sigma)\left\| \left(\bX^\top \bX\right)^{\odot (\ell+1)}-\Id \right\|_F.
    \end{align}
    \end{remark} 
Proposition \ref{prop:pkrr} can be viewed as an extension of \cite[Theorem 2.1]{el2010spectrum} and \cite[Lemma C.7]{donhauser2021rotational} for a specific inner-product kernel $\bK$ induced from the random CK with Gaussian weights, although \cite{el2010spectrum} and \cite{donhauser2021rotational} considered general rotational invariant random kernels. Our result reveals that we can simply employ such a truncated kernel to approximate the nonlinear kernel because of the $\ell$-orthonormal property in Assumption~\ref{assump:asymptotics}. In the proof of \cite{donhauser2021rotational}, the authors verified that such property holds for random data with high probability. The same form of $\bK$ has also been studied by \cite{liang2020multiple} for the ridgeless regression on some random data $\bX$ under the polynomial regime ($n\asymp d^\alpha$). Under a stronger regularity assumption on the kernel function, the authors first applied Taylor expansion to get truncated kernel $\bK_\ell$, then took the Gram-Schmidt process to obtain an orthogonal polynomial basis, which implied a sharper bound on the generalization error for random datasets.
\end{remark}

\subsection{Concentrations of the RFRR When  \texorpdfstring{$N\gg n$}{Nlaregen}}\label{sec:K-K_N}
We first consider a two-layer neural network at random initialization defined in \eqref{eq:1hlNN} and estimate the performance of random feature ridge regression in the ultra-high dimensional limit where $N\gg n$. 
We focus on the linear regression with respect to $\v \theta\in\R^{N}$ for predictors of the form
$f_{\v \theta}(\bX):=\frac{1}{\sqrt{N}}\v \theta^\top \sigma\left(\bW\bX\right)$,
with training data $\bX\in\R^{d\times n}$ and training labels $\v y\in\R^{ n}$. The loss function of the ridge regression with a ridge parameter $\lambda\geq  0$ is defined by
\begin{align}\label{def:lossfunction}
    L(\v \theta):=\frac{1}{n}\|f_{\v \theta}(\bX)^\top-\v y\|^2+\frac{\lambda }{n}\|\v \theta\|^2.
    \end{align}
The minimizer of \eqref{def:lossfunction} denoted by $ \hat{\v\theta}:=\arg\min_{\v\theta} L(\v\theta)$ has an explicit expression
\[
    \hat{\v\theta}= \frac{1}{\sqrt{N}}\bPhi\left(\frac{1}{N}\bPhi^\top \bPhi+\lambda\Id\right)^{-1}\v y,\]
where $\bPhi $ is defined in \eqref{eq:Y}. 
The optimal predictor for this RFRR with respect to the loss function in \eqref{def:lossfunction} is given by
\begin{align}
\hat{f}_{\lambda}^{(\RF)}(\bx):=&\frac{1}{\sqrt{N}}\hat{\v\theta}^\top \sigma\left(\bW\bx\right)
,\label{eq:RFridge}
\end{align}
where we define an empirical kernel $\bK_N(\cdot,\cdot):\R^{d} \times \R^{d} \to \R$ as $\bK_N(\bx,\bz):=\frac{1}{N}\sigma(\bW\x)^\top\sigma(\bW\z),$
and the $n$-dimension row vector is given by $\bK_N(\x,\bX)=[\bK_N(\x,\x_1),\ldots,\bK_N(\x,\x_n)]$. 

Analogously, consider any kernel function $\bK(\cdot,\cdot): \R^{d} \times \R^{d} \to \R$ defined in \eqref{eq:def_K}. Similar to \eqref{eq:RFridge}, the optimal kernel predictor with ridge parameter $\lambda$ for kernel ridge regression is given by 
\begin{align}\label{eq:Kridge}
\hat{f}_{\lambda}^{(\KK)}(\x):=\bK(\x,\bX) (\bK+\lambda\Id)^{-1} \v y.
\end{align}
 See \cite{rahimi2007random,avron2017random,liang2020just,jacot2020implicit,liu2021kernel,bartlett2021deep} for additional descriptions about KRR. 
 
We compare the behavior of the two different predictors  $\hat{f}_{\lambda}^{(\RF)}(\x)$ in \eqref{eq:RFridge} and $\hat{f}_{\lambda}^{(\KK)}(\x)$ in \eqref{eq:Kridge} with the kernel  $\bK$  defined in \eqref{eq:def_K}. As $N$ is sufficiently large, the empirical kernel $\bK_N$ defined in \eqref{eq:K_N(X)} will concentrate around its expectation \eqref{eq:def_K}. 
From \eqref{eq:RFridge} and \eqref{eq:Kridge}, the predictors of RFRR and KRR are determined by $\bK_N$ and $\bK$, respectively. Therefore, our concentration inequality will help us conclude that the performances of these two predictors are also close to each other as long as the width $N$ is sufficiently larger than sample size $n$.
In the following subsections, we will show that the training error, cross-validations, and generalization error of RFRR can be approximated by the corresponding quantities of KRR defined in \eqref{eq:Kridge} when  $N$ is sufficiently large.

\subsubsection{Training Error Approximation}\label{sec:train}
Denote the optimal predictors for the random feature and kernel ridge regressions on the training data $\bX$ with the ridge parameter $\lambda\geq 0$ by 
\begin{align*}
    \hat f_{\lambda}^{(\RF)}(\bX):=&\left(\hat f_{\lambda}^{(\RF)}(\x_1),\dots,\hat f_{\lambda}^{(\RF)}(\x_n)\right)^\top,\\
    \hat f_{\lambda}^{(\KK)}(\bX):=&\left(\hat f_{\lambda}^{(\KK)}(\x_1),\dots, \hat f_{\lambda}^{(\KK)}(\x_n)\right)^\top,
\end{align*}
respectively.  We first compare the training errors for these two predictors. Let the \textit{training errors} (empirical risks) of these two predictors be 
\begin{align}
 E_{\train}^{(\KK,\lambda)}=&\frac{1}{n}\|\hat{f}_{\lambda}^{(\KK)}(\bX)-\v y\|_2^2,
\label{eq:Etrain_Kn}\\
  E_{\train}^{(\RF,\lambda)}=&\frac{1}{n}\|\hat{f}_{\lambda}^{(\RF)}(\bX)-\v y\|_2^2 
.\label{eq:Etrain_K}
\end{align}

With high probability, the training error of a random feature model and the corresponding kernel model with the same ridge parameter $\lambda$ can be approximated as follows.

\begin{theorem}[Training error approximation]\label{thm:train_diff} 
Suppose that Assumptions \ref{assump:W}, \ref{assump:sigma}, and \ref{assump:asymptotics} hold. Then, with probability at least $1-N^{-2}$,  for any $\lambda\ge 0$, $N/\log^{2C_{\sigma}}(N)> C_1n$, and $n\geq n_0$, 
\begin{align}\label{eq:ETrain}
    \left| E_{\train}^{(\RF,\lambda)}-E_{\train}^{(\KK,\lambda)}\right|\leq  \frac{C_2\lambda^2\log^{C_{\sigma}} (N)\|\v y\|^2}{\sqrt{nN}} ,
\end{align}
where $C_1$ and $C_2$ are positive constants depending only on $\sigma$.
\end{theorem}

 Our bound \eqref{eq:ETrain} provides a non-asymptotic estimate on the training error approximation,  including the case when $\lambda= 0$.
From \eqref{eq:ETrain}, assuming $y_i=O(1)$ for all $i\in [n]$, we can conclude that the training error \eqref{eq:Etrain_Kn} concentrates around \eqref{eq:Etrain_K} as long as $N/\log^{2C_{\sigma}}(N) \gg n$. This result does not rely on the distribution of the data $\bX$ and how we generate the labels $\by$.

The random matrix tool we employ to prove Theorem~\ref{thm:train_diff} is a \textit{normalized} kernel matrix concentration inequality (Proposition~\ref{prop:kernel_con} in Appendix~\ref{sec:proof_310}).
In contrast to other kernel random matrix concentration results with deterministic $\bX$ in \cite{louart2018random,wang2021deformed}, a crucial property of our concentration inequality is that it does not depend on $\norm \bX$, which guarantees an $o(1)$ approximation error in \eqref{eq:ETrain} as long as $N/\log^{2C_{\sigma}}(N)\gg n$.  

\subsubsection{Cross-validations Approximation}\label{sec:cv}
 
 In the overparameterized regime,  the training error approximation in Theorem~\ref{thm:train_diff} does not directly imply a good approximation of the generalization, but the above analysis of training errors assists us in getting similar approximations on cross-validations of RFRR. 
 Cross-validation (CV) is a common method of model selection and parameter tuning in practice. Especially when practitioners have no access to the data distributions, one can employ CV to approximate the generalization errors of the model \cite{patil2022estimating,jacot2020kernel}.  For more background on cross-validations, we further refer to \cite{arlot2010survey}.

In this subsection,  we focus on leave-one-out cross-validation (LOOCV) and generalized cross-validation (GCV) for the predictors $ \hat f_{\lambda}^{(\RF)}$ and $ \hat f_{\lambda}^{(\KK)}$. Following \cite{hastie2009elements}, LOOCV is defined by 
\begin{equation}\label{eq:CV}
\begin{aligned}
     \CV_n^{(\KK,\lambda)}:=&\frac{1}{n}\sum_{i=1}^n (y_i-\hat{f}_{\lambda,-i}^{(\KK)}(\bx_i))^2,\\
     \CV_n^{(\RF,\lambda)}:=&\frac{1}{n}\sum_{i=1}^n (y_i-\hat{f}_{\lambda,-i}^{(\RF)}(\bx_i))^2,
\end{aligned}
\end{equation}
where $\hat{f}_{\lambda,-i}^{(\KK)}$ and $\hat{f}_{\lambda,-i}^{(\RF)}$ are KRR and RFRR estimators, respectively, on training data set $\bX$ with the data point $\bx_i$ removed. For simplicity, denote $\bK_{\lambda}=\bK +\lambda \Id$ and $\bK_{N,\lambda}=\bK_N +\lambda \Id$. With Schur complement, we can obtain the ``shortcut'' formulae for LOOCV as
\begin{align}
     \CV_n^{(\KK,\lambda)}= & \frac{1}{n}\by^\top \bK_{\lambda}^{-1}\bD^{-2} \bK_{\lambda}^{-1}\by,\label{eq:shortcut}\\
      CV_n^{(\RF,\lambda)}= & \frac{1}{n}\by^\top \bK_{N,\lambda}^{-1}\bD_N^{-2} \bK_{N,\lambda}^{-1}\by,\label{eq:shortcut_N}
\end{align}
where $\bD$ and $\bD_N$ are diagonal matrices with diagonals $[\bD]_{ii}=[\bK_{\lambda}^{-1}]_{ii}$ and $[\bD_N]_{ii}=[\bK_{N,\lambda}^{-1}]_{ii}$, for $i\in[n]$ respectively. The derivations of \eqref{eq:shortcut} and \eqref{eq:shortcut_N} are given in Lemma~\ref{lemm:shortcut} of Appendix~\ref{sec:LOOCV}.

Under certain assumptions, we expect $[\bD]_{ii}$ and $[\bD_N]_{ii}$ to concentrate around $\tr\bK_{\lambda}^{-1}$ and $\tr \bK_{N,\lambda}^{-1}$ respectively. Therefore, as an approximation of LOOCV, we define GCV
\begin{equation}\label{eq:GCV}
\begin{aligned}
      \GCV_n^{(\KK,\lambda)}:= &\left(\lambda\tr (\bK+\lambda\Id)^{-1}\right)^{-2}E_{\train}^{(\KK,\lambda)},\\
      \GCV_n^{(\RF,\lambda)}:=&\left(\lambda\tr (\bK_N+\lambda\Id)^{-1}\right)^{-2}E_{\train}^{(\RF,\lambda)}.
\end{aligned}
\end{equation}
For linear ridge regression models \cite{hastie@2022surprises}, the such approximation is done by applying random matrix theory to replace $\bD_{ii}$ with $\tr \bK_{\lambda}^{-1}$ and $[\bD_N]_{ii}$ with $\tr \bK_{N,\lambda}^{-1}$ in \eqref{eq:shortcut} and \eqref{eq:shortcut_N}, respectively. 

Since these cross-validation estimators are determined by training errors, with Theorem~\ref{thm:train_diff}, we obtain the concentrations of LOOCV and GCV.  Theorem~\ref{thm:cv_diff} reveals that under the ultra-wide regime, i.e., $N/\log^{2C_{\sigma}} N\gg n$, GCV and CV estimators of RFRR are close to the corresponding cross-validations of KRR, respectively.

\begin{theorem}[LOOCV and GCV approximations]\label{thm:cv_diff} 
Under the same assumptions as Theorem~\ref{thm:train_diff}, with probability at least $1-N^{-2}$,  for any $\lambda \ge  0$, when $N/\log^{2C_{\sigma}}(N)\geq C(1+\lambda^2)n$ and  $n\geq n_0$,
{\small
\begin{align} 
    \left|  \GCV_n^{(\KK,\lambda)}- \GCV_n^{(\RF,\lambda)}\right|\leq&    \frac{c(1+\lambda^4)\log^{C_{\sigma}} (N)\norm{\by}^2}{\sqrt{nN}} \label{eq:GCV_K}\\
    \left|  \CV_n^{(\KK,\lambda)}- \CV_n^{(\RF,\lambda)}\right|\leq & \frac{c(1+\lambda^4) \log^{C_{\sigma}} (N)\|\v y\|^2}{\sqrt{nN}}\label{eq:CV_K}
\end{align}
}%
where $C, c>0$ are constants depending only on $\sigma$.
\end{theorem}
 
The LOOCV and GCV of the linear model have been analyzed by   \cite{liu2019ridge,xu2021consistent,hastie@2022surprises,patil2022estimating,wei2022more}. As shown by \cite{hastie@2022surprises}, the advantage of LOOCV and GCV is that the optimal ridge parameter tuned by CV is asymptotically the same as the optimal ridge parameter in the high dimensional case. Unlike the results mentioned above, Theorem~\ref{thm:cv_diff} does not require any assumption on data distribution, which opens the door to studying LOOCV and GCV on more general datasets. 

In \cite{jacot2020kernel}, GCV is also called Kernel Alignment Risk Estimator (KARE), and the authors verified that GCV could be used to approximate the generalization error for KRR under a Gaussian universality hypothesis.
In addition, \cite{wei2022more} proved that GCV is a good approximation of the generalization error of the linear ridge regression model when a  local law for data distribution holds. This may imply that $\GCV_n^{(\KK,\lambda)}$ also asymptotically approaches the generalization error of KRR when the deterministic matrix $\bK(\bX,\bX)$ satisfies a local law property. This suggests that the concentrations in Theorem~\ref{thm:cv_diff} could be useful in approximating the generalization error of RFRR. Notably, \cite{wei2022more} considered general datasets under an anisotropic local law hypothesis, while our deterministic data only possesses some orthogonal structures. The proof of Theorem~\ref{thm:cv_diff} in  Appendix \ref{sec:LOOCV} opens a new avenue for analyzing LOOCV and GCV for kernel regression \cite{patil2022estimating}. Following \cite{wei2022more}, as a future research direction, we also expect that the GCV estimator of RFRR will converge to its generalization error under certain extra conditions.

\subsubsection{Generalization Error Approximation}

Different from the controls of in-sample prediction risks and cross-validations in Sections~\ref{sec:train} and \ref{sec:cv}, to investigate the generalization error, we introduce further assumptions on the model and the target function under a student-teacher model. The student-teacher model has been investigated in recent works \cite{gerace2020generalisation,dhifallah2020precise,hu2020universality,goldt2020modeling,loureiro2021learning,lin2021causes,damian2022neural,ba2022high}. Since all the data points $\bx_i$ are deterministic, our model is a fixed design rather than random design \cite{hsu2012random}. 

Denote an unknown teacher function by $f^*: \mathbb R^{d}\to \R$. The training labels are generated by
$
    \v y=f^*(\bX)+\boldsymbol \varepsilon,
$
where $ f^*(\bX)=(f^*(\x_1),\ldots,f^*(\x_n))^\top,$  and $\boldsymbol \varepsilon\sim \mathcal N (\mathbf{0},\sigma_\eps^2\Id)$.  
We impose the following assumptions.
\begin{assumption}\label{assump:target}
Assume that the target function is a nonlinear function with one neuron defined by $f^*(\bx)=\tau(\langle\bbeta,\bx\rangle),$
where the random vector $\bbeta\sim \cN(0, \Id_d)$ and $\tau\in L^4(\mathbb R,\Gamma)$. Suppose that $\zeta_k(\tau)\neq 0$ as long as $\zeta_k(\sigma)\neq 0$, for $0\le k\le \ell$. Training labels are given by $\by=\tau(\bX^\top\bbeta)+\beps \in \mathbb R^n$. 
\end{assumption}

In particular, such an assumption includes the case when  $\sigma$ and $\tau$ are the 
same activation function.

\begin{assumption}\label{assump:testdata}
Suppose  the test data $\bx\in\R^{d}$   satisfies  almost surely,  $\norm{\bx}=1$ and 
\begin{align}
\lim_{n\to\infty }\sqrt{n}\left\| (\bX^\top \bx)^{\odot (\ell+1)}\right\|_2=0.\label{asump:asym_test}
\end{align}  
\end{assumption}

Assumption \ref{assump:testdata} of the test data $\x$ guarantees similar statistical behavior as the training data points in $\bX$, but we do not impose any specific assumption on its distribution. It is promising to utilize such assumption further to handle statistical models with real-world data \cite{liao2020ARM,seddik2020random}.

Assumption \ref{assump:testdata}  holds with high probability in many cases when $\bx_1,\ldots,\bx_n$ are i.i.d.\ samples from some distribution $\cP$: e.g. $\text{Unif}(\mathbb{S}^{d-1})$ and  $\text{Unif}\left(\{\frac{-1}{\sqrt{d}},\frac{+1}{\sqrt{d}}\}\right)$ with $n\asymp d^{\alpha}$ and $\ell=2\lfloor\alpha \rfloor$; an arbitrary distribution  such that \eqref{asump:asym_test} holds almost surely through reject sampling \cite{casella2004generalized}; or an empirical distribution $\mu_{\hat n}$ where $\mu_{\hat n}=\frac{1}{\hat n} \sum_{i=1}^{\hat n}{\delta_{\hat{\bx}_i}}$,
and $\hat{\bx}_1,\dots, \hat{\bx}_{\hat n} $ are deterministic unit vectors such that \eqref{asump:asym_test} holds for each $\hat{\bx}_i, i\in [\hat n]$.

For any predictor denoted by $\hat{f}$, define the \textit{generalization error} (also called test error) to be the following conditional expectation 
\begin{align}\label{eq:def_testerror}
    \mathcal L(\hat f):=\E\left[|\hat f(\x)-f^*(\x)|^2\;\middle|\;\bX \right],
\end{align} 
where the expectation is taken over noise $\beps$, test data $\bx$, and signal $\bbeta$. 
Since the dataset $\bX$ is deterministic in our setting, the conditional expectation in \eqref{eq:def_testerror} becomes 
 $\mathcal L(\hat f)=\E[|\hat f(\x)-f^*(\x)|^2]$. Analogously to the linear case from \cite{ali2019continuous}, this turns out to be the \textit{Bayes risk} for out-of-sample predictors. Viewing $\bbeta$ as a random signal in the teacher model allows us to get a sharper bound of the generalization error in Theorem~\ref{thm:test_diff} below.

Under Assumption \ref{assump:testdata}, let $n_1$ be the smallest integer such that for all $n\geq n_1$,
 \begin{align}\label{assumption:test_orthogonal}
 \sup_{i\in [n]} |\langle \bx, \bx_i \rangle| \leq \frac{1}{\sqrt{2}}, \quad 
 \left\| (\bX^\top \bx)^{\odot (\ell+1)}\right\|_2 \leq  \frac{\sigma_{>\ell}^2}{4\|\sigma\|_4^2}.
 \end{align}
The following approximation holds for the test error between a random feature predictor and the corresponding kernel predictor in ridge regressions.
\begin{theorem}[Generalization error approximation]\label{thm:test_diff}  
Suppose Assumptions \ref{assump:W}, \ref{assump:sigma},  \ref{assump:target}, and \ref{assump:testdata} hold.  
Then, with probability at least $1-\log^{-C_{\sigma}}(N)$, for any $N/\log^{2{C_{\sigma}}} N\geq C_1(1+\lambda^2) n$, $n\geq \max\{n_0,n_1\}$, the difference between test errors of RFRR and KRR satisfies 
\begin{equation}\label{eq:test_diff}
    \left| \mathcal L(\hat{f}_{\lambda}^{(\RF)}(\x))-\mathcal L (\hat{f}_{\lambda}^{(\KK)}(\x))\right| \leq C_2(1+\lambda)\log^{C_{\sigma}}(N)\sqrt{\frac{n}{N}},
\end{equation}
for any $\lambda\geq 0$,  where constant $C_1$ depends only on $\sigma$, and  positive constant $C_2$ depends only on $\sigma,\tau$ and $\sigma_{\eps}$.
\end{theorem}
 When the width $N/\log^{2 C_{\sigma}}(N)\gg n$, the right-hand side of \eqref{eq:test_diff} is vanishing. In other words, RFRR has the same generalization error as KRR for ultra-wide neural networks. Notice that Theorem~\ref{thm:test_diff} covers the ridge-less regression case when $\lambda=0$.

\subsection{Approximation of KRR by a Polynomial KRR}
 In this subsection, we study a polynomial kernel ridge regression (PKRR) induced by the polynomial kernel $\bK_\ell$ in \eqref{def:Phi0}. We define an inner-product kernel by
\begin{equation*}
    \bK_{\ell}(\x,\z):= 
    \begin{cases} 
      \|\sigma\|_2^2, & \text{if } \x=\z \\
      \sum_{k=0}^\ell \zeta_k^2(\sigma)(\x^\top\z)^k, & \text{otherwise},
   \end{cases}
\end{equation*}
for any $\bx,\bz\in \mathbb R^d$. The parameter $\ell $ defined by Assumption~\ref{assump:asymptotics}, is determined by orthogonality among different data points in the training set. In practice, it is hard to implement the expected kernel $\bK$, whereas this truncated kernel $\bK_{\ell}$ with finite many parameters is a simpler model for implementation and theoretical analysis. Similarly, with \eqref{eq:RFridge} and \eqref{eq:Kridge}, the predictor for kernel regression with respect to $\bK_{\ell}$ is denoted by
\begin{align} \label{eq:Kl_ridge}
\hat{f}_{\lambda}^{(\ell)}(\x):=\bK_{\ell}(\x,\bX) (\bK_{\ell}+\lambda\Id)^{-1} \v y,
\end{align}
where, by an abuse of notation, we use $\bK_{\ell}$ to denote the $n\times n$ polynomial kernel matrix  $\bK_{\ell}(\bX,\bX)$. For simplicity, denote $\bK_{\ell, \lambda}:=\bK_{\ell}+\lambda\Id$ for any $\lambda\ge 0$.

Based on Proposition \ref{prop:pkrr}, we show that the performances of KRR with kernel $\bK$ can be approached by the performances of $\hat{f}_{\lambda}^{(\ell)}$. Denote the training error, the CV, GCV and test error for $\bK_{\ell}$ as $E_{\train}^{(\ell,\lambda)}$,  $\CV_n^{(\ell,\lambda)}$, $\GCV_n^{(\ell,\lambda)}$, and $\mathcal L(\hat{f}_{\lambda}^{(\ell)}(\x))$ respectively. By replacing $\bK$ by $\bK_\ell$, we can define these estimators of the PKRR similarly with \eqref{eq:Etrain_K}, \eqref{eq:CV}, and \eqref{eq:GCV}.  
Denote $\tilde\bX=[\bX,\bx]\in\R^{d\times (n+1)}$  the concatenation of training and test data points. Denote
   \[\Delta_{\ell}=\left\| \left( \bX^\top \bX\right)^{\odot \ell+1}-\Id\right\|_F ,\quad  \tilde\Delta_\ell=\left\| \left(\tilde \bX^\top \tilde \bX\right)^{\odot \ell+1}-\Id\right\|_F.\] 
Under \eqref{assump:asymptotic_l} and \eqref{asump:asym_test}, we have $\Delta_\ell, \tilde\Delta_\ell=o_n(1).$

\begin{theorem}\label{thm:pkrr}
Suppose  Assumptions \ref{assump:sigma} and \ref{assump:asymptotics}   hold. Then  for $n\geq n_0$,
{\small\begin{align}
    \left| E_{\train}^{(\ell,\lambda)}-E_{\train}^{(\KK,\lambda)}\right|\leq  &\frac{C_1\lambda^2 \|\by\|^2}{n} \Delta_\ell,\label{eq:Kl_train}\\
       \left| \GCV_n^{(\ell,\lambda)}-\GCV_n^{(\KK,\lambda)}\right|\leq  &\frac{C_1 (1+\lambda^4) \norm{\by}^2}{n}\Delta_\ell, \label{eq:GCVn}\\
    \left| \CV_n^{(\ell,\lambda)}-\CV_n^{(\KK,\lambda)}\right|\leq  & \frac{C_1 (1+\lambda^4) \norm{\by}^2}{n}\Delta_\ell,\label{eq:CVn}
\end{align}
}%
where $C_1>0$ depend only on $\sigma$.
Furthermore, with Assumptions \ref{assump:target} and \ref{assump:testdata}, for $n\geq \max\{n_0,n_1\}$, the generalization errors of KRR and PKRR satisfy that 
\begin{equation}\label{eq:Kl_test}
     \left| \mathcal L(\hat{f}_{\lambda}^{(\ell)}(\x))-\mathcal L (\hat{f}_{\lambda}^{(\KK)}(\x))\right| \le C_2(1+\lambda)\tilde \Delta_\ell,
\end{equation}for  some constant $C_2>0$ only depending on $\sigma, \tau, \sigma_\eps$.
\end{theorem}

Based on the definition of $\ell$ in \eqref{assump:asymptotics}, if  the training labels satisfy $\norm \by^2=O(n)$, the left-hand sides of \eqref{eq:Kl_train}-\eqref{eq:Kl_test} are all vanishing as $n\to\infty$. Combining the concentration between RFRR and KRR in Section~\ref{sec:K-K_N}, we can now conclude that, in terms of training/test errors and cross-validations, the performance of the RFRR is close to the performance of PKRR defined in \eqref{eq:Kl_ridge} with high probability as long as  $N/\log^{2C_{\sigma}} N\gg n$ and $n\to\infty$. Therefore, the behaviors of the RFRR generated by ultra-wide neural networks can be characterized by a much simpler PKRR induced by the expected kernel $\bK$. For \eqref{eq:Kl_test}, we can actually verify the estimators $\hat{f}_{\lambda}^{(\KK)}(\x)$ and $\hat{f}_{\lambda}^{(\RF)}(\x)$ are polynomials of $\bx$ with degree at most $\ell$, which is analogous to the second part of \cite[Theorem C.2]{donhauser2021rotational}.  Similar results on neural tangent feature regression are proved by \cite{montanari2020interpolation} for uniform spherical distributed data.
Due to this simplification, we can further obtain a lower bound of the generalization error of RFRR in the next subsection.

\subsection{Polynomial Approximation Barrier for RFRR}
The \textit{polynomial approximation barrier} refers to the case when an estimator $\hat{f}_\lambda$ cannot learn any polynomial with a degree larger than a certain threshold \cite{donhauser2021rotational}. This phenomenon has been shown in both RFRR and KRR \cite{mei2019generalization,ghorbani2021linearized,mei2021generalization,donhauser2021rotational} under specific data distribution assumptions, e.g., uniform
distributions on the unit sphere or hypercubes (or more general distributions with hypercontractivity assumptions and proper eigenvalue decays) and anisotropic distributions with covariance structures \cite{loureiro2021learning,gerace2022gaussian}.  

Define $P_{> \ell}: L^2(\mathbb R, \Gamma) \to L^2(\mathbb R, \Gamma) $ as the projection onto the span of Hermite polynomials defined in \eqref{eq:hermite} with degrees at least $\ell+1$. Specifically, recalling $\bbeta\sim \cN(0,\Id)$ and $\norm{\bx}=1$, we can get 
$\left( P_{> \ell}f^*\right)(\bx)=\sum_{k\geq \ell+1} \zeta_k(\tau)h_k(\bbeta^\top\bx),$
where  $\zeta_k(\tau)$ is defined by \eqref{eq:hermite_coeff}. Denote
$$
\|P_{> \ell}f^*\|_{2}^2 = \mathbb E_{\bx, \bbeta}\left( P_{>\ell}f^*(\bx)\right)^2  =\sum_{k\geq \ell+1} \zeta_k^2(\tau).
$$
In the following theorem, we prove that the polynomial approximation barrier for RFRR is related to the $\ell$-orthonormal properties of the training data. 
Theorem~\ref{thm:train_diff} and Theorem~\ref{thm:test_diff} verify that the RFRR achieves the same errors as KRR, as long as $N$ is sufficiently large. Meanwhile, Theorem~\ref{thm:pkrr} shows KRR can be further approximated by a simpler polynomial kernel model, whose degree $\ell$ is determined by the $\ell$-orthonormal property in \eqref{assump:asymptotics}.
Combining these together, RFRR induced by an ultra-wide neural network is asymptotically equivalent to an $\ell$-degree PKRR, which naturally implies that RFRR is unable to learn any function with higher-degree terms consistently.

\begin{theorem}[Lower bound of the generalization error for RFRR]\label{thm:lower_bound}
Under the assumptions of Theorem~\ref{thm:test_diff}, with  probability at least $1-\log^{-C_\sigma}(N)$, 
when $N/\log^{2C_{\sigma}} N\gg n,$ 
\begin{align}
 \mathcal L(\hat{f}_{\lambda}^{(\RF)})
 &\ge \|P_{> \ell}f^*\|^2_{2}+\sigma_{\beps}^2\E_\bx\left[\bK_{m,\ell}^\top\bK_{\lambda,\ell}^{-2}\bK_{m,\ell}\right]-o_n(1)\geq \|P_{> \ell}f^*\|^2_{2}-o_n(1),\label{eq:simplified_asym}
\end{align}
where $\bK_{m,\ell}:=\bK_\ell(\bX,\bx), \bK_{\ell,\lambda}:=\lambda\Id+\bK_\ell(\bX,\bX)$.
\end{theorem}

In Theorem~\ref{thm:lower_bound}, we specifically consider a test data point  with the $\ell$-orthonormal property. This simplifies the teacher model in Assumption~\ref{assump:target} since $f^*(\bx)$ has the same  in distribution as $\tau(\xi)$ for $\xi\sim\cN(0,1)$. Therefore, Theorem~\ref{thm:lower_bound} reveals that RFRR predictor $\hat{f}_{\lambda}^{(\RF)}$ cannot learn the higher degree terms in the Hermite expansion of target function $\tau$. This threshold $\ell$ is determined by the $\ell$-orthonormal property of $\bX$ in \eqref{assump:asymptotics}. The more orthogonal the data points in $\bX$ are, the lower degree of Hermite polynomials this RFRR predictor can learn consistently.

\begin{remark}[The variance term]
    The second term in the first lower bound of \eqref{eq:simplified_asym} is related to the variance term in the generalization error of PKRR. This term can be further simplified based on some additional assumptions on the data distribution. Specifically, \cite[Theorem 2]{liang2020multiple} validated that for sub-Gaussian data, 
    $$\Tr\bK_{\ell,\lambda}^{-1}\E_{\bx}[\bK_{\ell}(\bX,\bx)\bK_{\ell}(\bx,\bX)] \bK_{\ell, \lambda}^{-1}\lesssim \frac{d^\alpha}{n}+\frac{n}{d^{\alpha+1}}
    $$ with high probability, when $d^\alpha\log d\lesssim n\lesssim d^{\alpha+1}$. Hence, this bound is vanishing in this polynomial regime (see also \cite[Secion 4]{bartlett2021deep}). In contrast, under the critical regime $n\asymp d^\alpha$, this variance term, in KRR of any inner-product kernel for uniform spherical distribution, is provably non-degenerate, determined by the Marchenko-Pastur distribution, and may even result in a peak in the prediction curve \cite{misiakiewicz2022spectrum,hu2022sharp,xiao2022precise}.
\end{remark}

\begin{remark}[Comparison to previous work with random dataset]

The lower bounds in Theorem~\ref{thm:lower_bound} exhibit the limitation of the RFRR and KRR: \eqref{eq:simplified_asym} implies RFRR estimator cannot learn any higher degree polynomials. This is useful when we aim to show that some estimator is superior to this RFRR 
 estimator \cite{ba2022high,damian2022neural}. Compared with the results of \cite{ghorbani2021linearized,mei2021generalization}, our results cover more general training datasets for RFRR, though it is not optimal in some specific circumstances (see Remark~\ref{rmk:compare_to_random}), and we only address the single-neuron student-teacher model. Since we study RFRR on a general dataset without any data distribution assumptions, we cannot obtain a more precise characterization of the generalization error as the results by \cite{mei2021generalization}. On the other hand, \cite{donhauser2021rotational} exhibited a lower bound $\|P_{>(2\lfloor2\alpha \rfloor) }f^*\|^2_{2}$ on the generalization error for kernel ridge regression with a general rotational invariant kernel (which is $\|P_{>(2\lfloor\alpha \rfloor) }f^*\|^2_{2}$ when data $\bx$ has unit length),  where the dataset $\bX\in\R^{d\times n}$ is random and satisfies $n\asymp d^\alpha$.  Under more general assumptions on the dataset $\bX$, we obtain a similar lower bound  $\|P_{>(2\lfloor\alpha \rfloor) }f^*\|^2_{2}$ from \eqref{eq:simplified_asym} for both RFRR and its corresponding KRR.
\end{remark}

\section{Simulations}\label{sec:simulate}
\begin{figure}[ht]  
\centering
\begin{minipage}[t]{0.488\linewidth}
\centering 
{\includegraphics[width=0.99\textwidth]{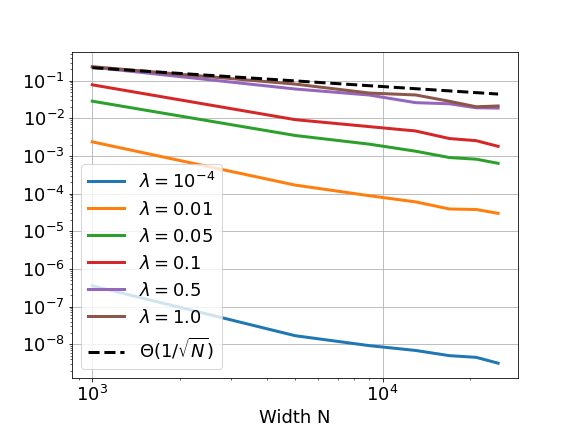}} \\
{\small (a) Training error}
\end{minipage} 
\begin{minipage}[t]{0.488\linewidth}
\centering
{\includegraphics[width=0.99\textwidth]{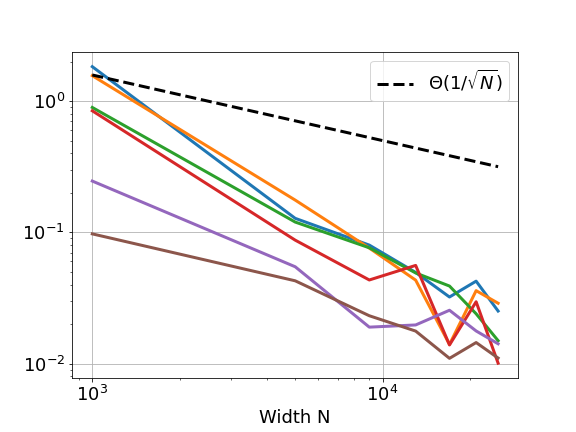}} \\
{\small (b) LOOCV error}  
\end{minipage} 
\begin{minipage}[t]{0.488\linewidth}
\centering 
{\includegraphics[width=0.99\textwidth]{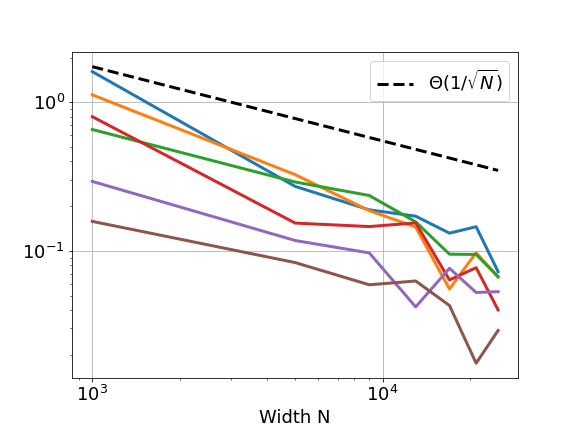}} \\
{\small (c) GCV error}
\end{minipage} 
\begin{minipage}[t]{0.488\linewidth}
\centering
{\includegraphics[width=0.99\textwidth]{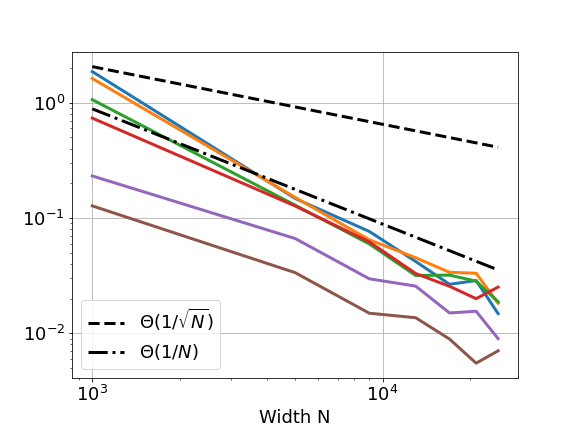}} \\
{\small (d) Test error}  
\end{minipage} 
\vspace{-0.15cm}
\caption{\small
Differences between KRR and RFRR with various ridge parameters $\lambda$ for (a) training errors, (b) LOOCV errors, (c) GCV errors, and (d) generalization errors. Data $\bX$ is i.i.d.\ sampled from uniform distribution $\text{Unif}(\mathbb{S}^{d-1})$ with $d=n = 500$ and training label noise $\sigma_\eps = 0.6$. We repeat each experiment with 7 trials to average. The target function $\tau$ is Softplus.}   
\label{fig:gauss}  
\end{figure}
In Figure~\ref{fig:gauss}, we empirically verify the concentration bounds we derived in Theorems~\ref{thm:train_diff}, \ref{thm:cv_diff}, \ref{thm:test_diff} and \ref{thm:pkrr} using i.i.d. random data $\bX$, where each data point is sampled from $\text{Unif}(\mathbb{S}^{d-1})$ with $d=n = 500$. As the width $N$ increases, we observe that the differences for training errors, LOOCV, GCV, and generalization errors between RFFR and KRR are all convergent with a rate of at least $1/\sqrt{N}$. The activation function is a polynomial $p(x):=h_0(x)+\frac{1}{\sqrt{6}}h_1(x)+\frac{1}{3}h_2(x)+\frac{1}{6}h_3(x)+\frac{2}{3}h_4(x)+\frac{1}{2}h_5(x)$. For KRR, we utilize the polynomial KRR $\bK_\ell$ defined by \eqref{def:Phi0} with $\ell=2$ for an approximation of the original $\bK$. Additional simulations on the synthetic datasets are presented in Appendix~\ref{sec:add_sim}. 

\begin{figure}[ht]  
\centering
\begin{minipage}[t]{0.488\linewidth}
\centering 
{\includegraphics[width=0.99\textwidth]{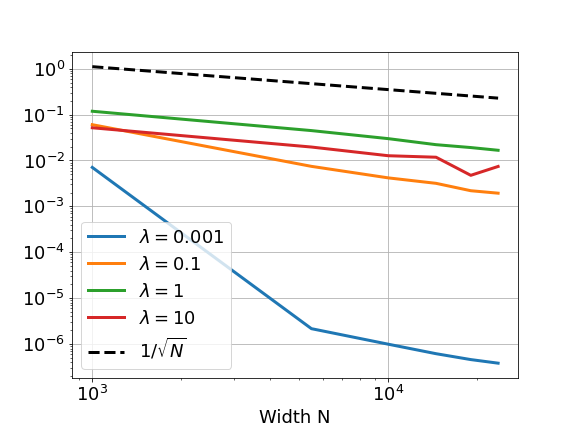}} \\
{\small (a) Training error}
\end{minipage} 
\begin{minipage}[t]{0.488\linewidth}
\centering
{\includegraphics[width=0.99\textwidth]{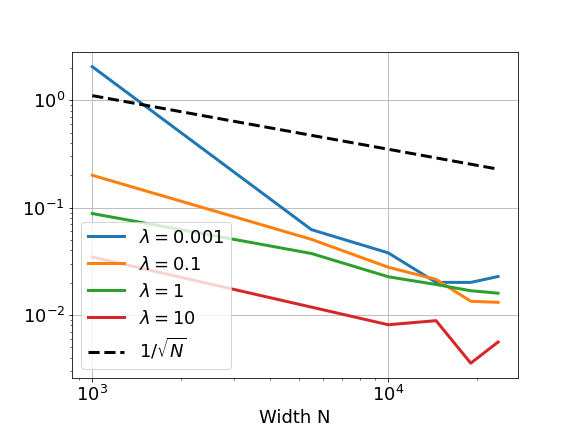}} \\
{\small (b) LOOCV error}  
\end{minipage} 
\begin{minipage}[t]{0.488\linewidth}
\centering 
{\includegraphics[width=0.99\textwidth]{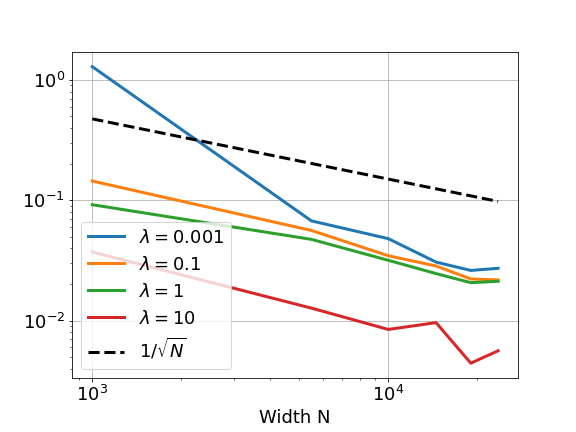}} \\
{\small (c) GCV error}
\end{minipage} 
\begin{minipage}[t]{0.488\linewidth}
\centering
{\includegraphics[width=0.99\textwidth]{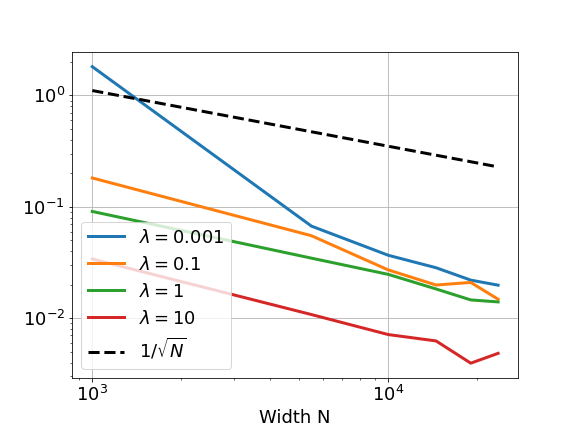}} \\
{\small (d) Test error}  
\end{minipage} 
\vspace{-0.15cm}
\caption{\small
Differences between KRR and RFRR with various ridge parameters $\lambda$ for (a) training errors, (b) LOOCV errors, (c) GCV errors, and (d) generalization errors. Data points in $\bX$ are randomly selected from CIFAR-10 with $d=800$, $n = 1000$ training samples, and without label noise. We repeat each experiment with 5 trials. The target function $\tau$ is the ReLU function.}   
\label{fig:cifar}  
\end{figure}
Analogously, we investigate the concentrations between RFRR and KRR on real-world data in Figure~\ref{fig:cifar}. We randomly select $d=800$ features for each data vector and $n=1000$ data points in the CIFAR-10 dataset. After normalizing the data points, we compare the performances of RFRR and KRR induced by the activation function $p(x)$. We observe that our theoretical concentration bound $1/\sqrt{N}$ derived from Section~\ref{sec:mainresults} is almost optimal in Figure~\ref{fig:cifar}. We expect to further explore which real-world datasets will empirically satisfy the $\ell$-orthonormal property defined in Assumption~\ref{assump:asymptotic_l} as a future direction.

\section{Conclusion} 
In this paper, we studied the behavior of random feature ridge regression in the overparameterized regime $(N\gg n)$ with a deterministic dataset under an $\ell$-orthonormal assumption. In our analysis, we proposed refined matrix concentration inequalities with relaxed assumptions and a convenient Hermite polynomial expansion of the nonlinear activation function. These approaches allow us to go beyond the linear regime \cite{wang2021deformed},
leading us to study any polynomial kernel approximation of RFRR and obtain new results for  general deterministic datasets. 

Our analysis has highlighted the impact of the degree of orthogonality among different input data points on the performance of RFRR in terms of training and generalization errors and cross-validation.
In addition, Hermite polynomial expansion of $\sigma$ is a universal way to precisely analyze RFRR induced by any two-layer neural networks with Gaussian random weights. As one-dimensional polynomials, they are easier to implement in practice compared to other orthogonal polynomial expansion approaches \cite{misiakiewicz2022spectrum,hu2022sharp,xiao2022precise,ghorbani2021linearized,mei2021generalization} that depend on both data and weight distributions for RFRR. We anticipate that our approach can also be applied to analyze other random kernel matrices, including the empirical NTK, from more general multi-layer neural networks with general deterministic datasets.

\subsubsection*{Acknowledgements}
Z.W. is partially supported by NSF DMS-2055340 and NSF DMS-2154099. Y.Z. is partially supported by  NSF-Simons Research Collaborations on the Mathematical and Scientific
Foundations of Deep Learning.  
This work was done in part while both authors were visiting the Simons Institute for the Theory of Computing during the summer of 2022.
 Z.W. would like to thank Denny Wu for his valuable suggestions and comments. Both authors thank Konstantin Donhauser and Yiqiao Zhong for their helpful discussions.

\bibliographystyle{alpha} 
\bibliography{ref.bib}

\appendix

\section{Addtional Simulations}\label{sec:add_sim}
As a complementary, Figure~\ref{fig:gauss1} shows the convergence rates for the differences in training errors, LOOCV errors, GCV errors, and generalization errors between RFRR and KRR. In this experiment, the data points are i.i.d.\ sampled from $\text{Unif}(\mathbb{S}^{d-1})$ with $d = 500$ and training samples $n=1000$. The activation function is a degree-5 polynomial $p(x)=h_0(x)+\frac{1}{\sqrt{6}}h_1(x)+\frac{1}{3}h_2(x)+\frac{1}{6}h_3(x)+\frac{2}{3}h_4(x)+\frac{1}{2}h_5(x)$, where Hermite polynomials are defined in Definition~\ref{eq:hermite}. As an approximation of the kernel $\bK$ generated by $\sigma(x)=p(x)$, we can consider $\bK_2 $ defined by 
\[\bK_2=\mathbf{1}\mathbf{1}^\top+\frac{1}{6}\bX^\top\bX+\frac{1}{9}(\bX^\top\bX)^{\odot 2}+\frac{26}{36}\Id.\]
We employ this simple kernel $\bK_2 $ to compute the performances of KRR and compare them with the performances of RFRR generated by $\sigma$ and \eqref{eq:K_N(X)}. In this simulation, we consider a teacher model defined by Assumption~\ref{assump:target} where $\tau$ is the Softplus function. Similarly with Figures~\ref{fig:gauss} and \ref{fig:cifar}, these results of the simulation match with our theorems in Section~\ref{sec:mainresults}.
\begin{figure}[ht]  
\centering
\begin{minipage}[t]{0.488\linewidth}
\centering 
{\includegraphics[width=0.99\textwidth]{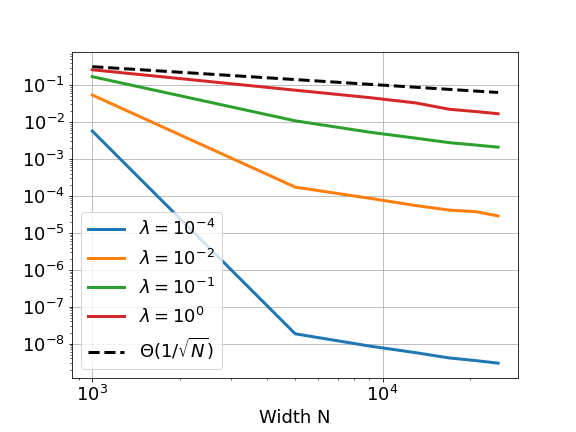}} \\
{\small (a) Training error}
\end{minipage} 
\begin{minipage}[t]{0.488\linewidth}
\centering
{\includegraphics[width=0.99\textwidth]{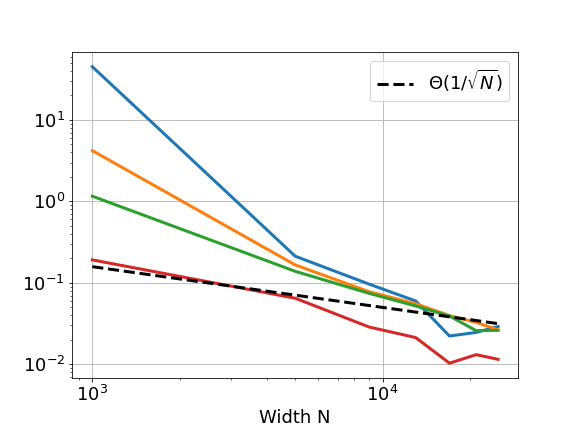}} \\
{\small (b) LOOCV error}  
\end{minipage} 
\begin{minipage}[t]{0.488\linewidth}
\centering 
{\includegraphics[width=0.99\textwidth]{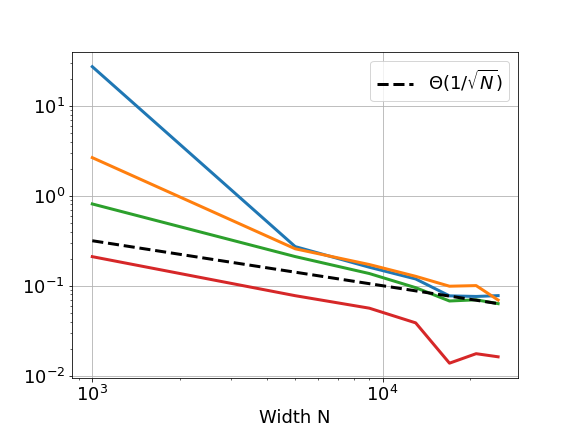}} \\
{\small (c) GCV error}
\end{minipage} 
\begin{minipage}[t]{0.488\linewidth}
\centering
{\includegraphics[width=0.99\textwidth]{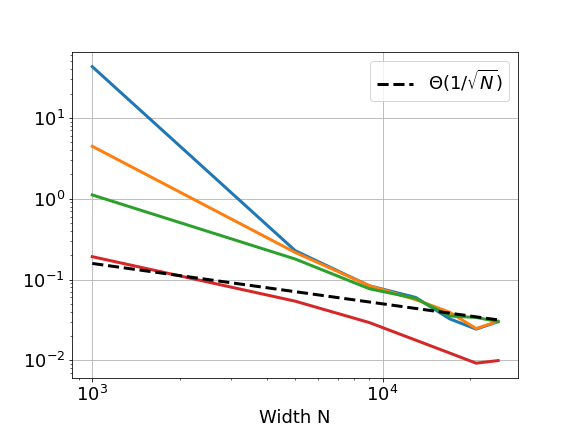}} \\
{\small (d) Test error}  
\end{minipage} 
\caption{\small
Differences between KRR and RFRR with various ridge parameters $\lambda$ in terms of (a) training errors, (b) LOOCV errors, (c) GCV errors, and (d) generalization errors. Here, data $\bX$ is sampled from $\text{Unif}(\mathbb{S}^{d-1})$ with $d = 500$, $n=1000$ and training label noise $\sigma_\eps = 0.3$. We repeat each experiment with 5 trials to average. The target function $\tau$ is Softplus function.}   
\label{fig:gauss1}  
\end{figure}
\section{Additional Notations and Definitions}
We denote $\Id$ as the identity matrix. Let $\bK_\lambda=\bK+\lambda\Id$ where $\bK$ is defined by \eqref{def:phi} and $\lambda\ge 0$ is the ridge parameter. Denote $\bK_{N,\lambda}=\bK_N+\lambda\Id$ where $\bK_N$ is \eqref{eq:K_N(X)}. Conventionally, let $\norm{\cdot}$ be the $\ell_2$-norm for vectors and $\ell_2\to\ell_2$ operator norm for matrices. Let $\preccurlyeq$ be the Loewner order for positive semi-definite matrices. For any matrix $\bA\in\R_{n\times n}$, $[\bA]_{i,j}$ denotes the $(i,j)$ entry of $\bA$, and $[\bA]_{[i,:]}$ denotes the $i$-th row of $\bA$ for any $i,j\in[n]$. Recall that the constant $\lambda_0=\frac{1}{2}\sigma_{>\ell}^2>0$. In the following proofs in Appendix~\ref{sec:proofs}, all the constants are universal and do not depend on $n,d,$ and $N$.

The following normalized Hermite polynomials are necessary for expanding $\sigma$ and approximating  $\bK$ by a polynomial kernel in Section~\ref{sec:approx_kernel} under Gaussian distributions.
\begin{definition}[Normalized Hermite polynomial]\label{eq:hermitepolynomial}
The $k$-th normalized Hermite polynomial is given by 
\begin{equation}\label{eq:hermite}
    h_k(x)=\frac{1}{\sqrt {k!}} (-1)^k e^{x^2/2} \frac{d^k}{dx^k} e^{-x^2/2}.
\end{equation}
These polynomials $\{h_k\}_{k=0}^{\infty}$ form an orthogonal basis of $L^2(\mathbb R, \Gamma)$, where $\Gamma$ denotes the standard Gaussian distribution. For any $\sigma_1,\sigma_2\in L^2(\mathbb R, \Gamma)$,  the inner product, with respect to the standard Gaussian measure, is defined by 
\[
    \langle \sigma_1,\sigma_2\rangle_{\Gamma} =\int_{-\infty}^{\infty} \sigma_1(x)\sigma_2(x) \frac{e^{-x^2/2}}{\sqrt{2\pi}}dx.
\]
\end{definition}  

Based on the definition, every function $\sigma \in L^2(\mathbb R, \Gamma)$ can be expanded as  $    \sigma(x)=\sum_{k=0}^{\infty}\zeta_k(\sigma)h_k(x)$,
where $\zeta_k(\sigma)$ is the $k$-th Hermite coefficient given by 
\begin{align}
    \zeta_k(\sigma)=\int_{-\infty}^{\infty} \sigma(x)h_k(x)\frac{e^{-x^2/2}}{\sqrt{2\pi}}dx, \label{eq:hermite_coeff}
\end{align}
and $ \|\sigma\|_2^2=\sum_{k=0}^{\infty} \zeta_k^2(\sigma)$. Moreover, we have   $ \langle h_k,h_j\rangle_{\Gamma}=\E[h_k(\xi)h_j(\xi)]=\delta_{j,k}$ for any $\xi\sim\cN(0,1)$ and $k,j\in\N$. For more properties of Hermite polynomials, see \cite{oymak2020toward,nguyen2020global}.

\section{Proofs of Main Results in Section~\ref{sec:mainresults}}\label{sec:proofs}
\subsection{Proof of Proposition \ref{prop:pkrr}}

By the Hermite polynomial expansion of $\sigma$, for $i,j\in [n]$, we have 
\[ \bK_{ij}=\E_{\w}[\sigma(\w_i^\top \bx_i)\sigma(\w_i^\top \bx_j)]=\sum_{k=0}^{\infty} \xi_k^2(\sigma) \langle \bx_i, \bx_j\rangle^k.
\]
Thus, we can expand this kernel as
\begin{align}
    \bK=\sum_{k=0}^{\infty} \xi_k^2(\sigma)\left(\bX^\top \bX\right)^{\odot k}, \quad \bK-\bK_{\ell}=\sum_{k=\ell+1}^{\infty } \xi_k^2(\sigma)\left(\left(\bX^\top \bX\right)^{\odot k}-\Id\right). \notag
\end{align}
Then by Cauchy's inequality, for $n\geq n_0$,
\begin{align*}
    \|\bK-\bK_{\ell}\|^2&\leq  \|\bK-\bK_{\ell}\|_F^2= \sum_{i\not= j} \left( \sum_{k=\ell+1}^{\infty} \xi_k^2(\sigma) \langle \bx_i,\bx_j\rangle^k\right)^2 \\
    &\leq \sum_{i\not=j} \left(\sum_{k=\ell+1}^{\infty} \xi_k^4(\sigma)\right) \left(\sum_{k=\ell+1}^{\infty}\langle \bx_i,\bx_j\rangle^{2k} \right)\\
    &\leq \|\sigma\|_4^4 \sum_{i\not=j} \frac{\langle \bx_i,\bx_j\rangle^{2\ell+2}}{1-\max |\langle \bx_i,\bx_j\rangle |^2}\leq 2\|\sigma\|_4^4 \left\| \left(\bX^\top \bX\right)^{\odot \ell+1}-\Id\right\|_F^2.
\end{align*}
Therefore,  from \eqref{assump:relaxed_l},
\begin{align*}
   \|\bK-\bK_{\ell}\|\leq   \sqrt{2}\|\sigma\|_4^2\left\| \left(\bX^\top \bX\right)^{\odot \ell+1}-\Id\right\|_F\leq \frac{1}{2\sqrt{2}} \sigma_{>\ell}^2.
\end{align*}

Since \[\left(\bX^\top \bX\right)^{\odot k}_{ij}= \langle \bx_i,\bx_j\rangle^k=\langle \bx_i\otimes \cdots \otimes \bx_i, \bx_j\otimes \cdots \otimes \bx_j\rangle, \]
 where $\x_i\otimes \cdots \otimes \x_i$ is the $k$-th tensor product of $\x_i$, $\left(\bX^\top \bX\right)^{\odot k}$ is positive semidefinite.  Then 
 \begin{align}
     \lambda_{\min}(\bK)\geq \lambda_{\min}(\bK_{\ell})-\|\bK-\bK_{\ell}\|
    &\geq  \sigma_{>\ell}^2 -  \frac{1}{2\sqrt{2}} \sigma_{>\ell}^2>  \frac{1}{2} \sigma_{>\ell}^2\equiv \lambda_0, \notag
 \end{align}
  Notice that $\lambda_0>0$ because $\sigma^2_{>\ell}>0$ from Assumption~\ref{assump:asymptotics}.

\subsection{Concentration Inequality for Normalized Random Kernel Matrices} \label{sec:proof_310}
Now we introduce the concentration inequality for $\bK_N$ in a normalized version, which is the cornerstone for proving Theorem~\ref{thm:train_diff}. 
Similar concentration results were also obtained in Theorem 3.2 of \cite{montanari2020interpolation} for the neural tangent kernel (NTK), where the data matrix $\bX$ is assumed to be uniformly random, and the activation function is assumed to have a polynomial growth rate, while we make no distribution assumption on $\bX$ and only assume $\|\sigma\|_4$ is finite. To consider a normalized version of the kernel matrices, we need to consider $\bK_{\lambda}^{-1}$. Under Assumption \ref{assump:asymptotics}, we use Proposition \ref{prop:pkrr} to make sure $\bK_{\lambda}$ is invertible when $\lambda=0$.  
 
\begin{prop}[Normalized random kernel matrix concentration]\label{prop:kernel_con}
Suppose that Assumptions \ref{assump:W}, \ref{assump:sigma}, and \ref{assump:asymptotics} hold. There exists some positive constants $C_1, C_2>0$ depending only on $\sigma$, such that for any  $N$ satisfying $N/\log^{2C_{\sigma}}(N)> C_1 n$  and any $\lambda\ge 0, n\geq n_0$, we have 
\begin{equation}\label{eq:kernel_con}
\norm{\bK_\lambda^{-\frac{1}{2}}\left(\bK_N-\bK\right)\bK_\lambda^{-\frac{1}{2}} } \leq  C_2 \log^{C_{\sigma}}(N)\sqrt{\frac{n}{N}}
\end{equation}
with probability at least $1-N^{-2}$, where $\bK_{\lambda}=\bK +\lambda \Id$.
\end{prop}
\begin{proof}
Denote $\tilde\sigma(x):=\sigma(x)\mathbf{1}_{|x|\le B}$, where $B$ is a parameter to be decided later. Define
\begin{align*}
    \bK_N =& \frac{1}{N} \sum_{i=1}^N\sigma(\w_i^\top \bX)^\top\sigma(\w_i^\top \bX) ,&~~&\bK = \E_{\bw} [\sigma(\w^\top \bX)^\top\sigma(\w^\top \bX)],\\
    \tilde\bK_N = &\frac{1}{N} \sum_{i=1}^N\tilde\sigma(\w_i^\top \bX)^\top\tilde\sigma(\w_i^\top \bX),&~~&\tilde\bK = \E_{\bw} [\tilde\sigma(\w^\top \bX)^\top\tilde\sigma(\w^\top \bX)].
\end{align*}
For simplicity, we denote $\tilde\bK_\lambda:=\tilde\bK+\lambda\Id$. 
Define \[\bH_i:=\frac{1}{N}\tilde\bK_\lambda^{-1/2}\tilde\sigma(\w_i^\top \bX)^\top\tilde\sigma(\w_i^\top \bX)\tilde\bK_\lambda^{-1/2}.\] 
Notice that Proposition~\ref{prop:pkrr} implies that $\norm{\bK_\lambda^{-1}}\le \lambda_0^{-1}$ for $\lambda\ge 0$. Firstly, based on the truncated function $\tilde\sigma(x)$, we have that for some universal constant $c>0$,
\begin{align}
    \Parg{\bK_N\neq \tilde\bK_N}\le~& \Parg{\max_{i\in [N],k\in[n]}|\bw_i^\top\bx_k|>B}
    \le  Nn\Parg{|\xi|>B}\le~ cNn\exp{(-B^2/2)},\label{eq:KNnotK}
\end{align}
where $\xi\sim\cN(0,1)$. 
Define the event by $A_i:=\{\bw:|\bw^\top\bx_i|\le B\}$ for $i\in[n]$. Entry-wisely, we have
\begin{align*}
   \left| [\bK-\tilde\bK]_{i,j}\right|=~&\left|\E_\bw[\sigma(\bw^\top\bx_i)\sigma(\bw^\top\bx_j)\mathbf{1}_{A_j^c\cup A_j^c}]\right|\\
   \le~&  \E_{\bw}[\sigma(\bw^\top\bx_i)^4]^{1/2}\E[\mathbf{1}_{A_j^c\cup A_j^c}]^{1/2}\\
   \le~ & \sqrt{2}\E [\sigma(\xi)^4]^{1/2}\Parg{A_i^c}^{1/2}\le C_0\exp{(-B^2/4)},
\end{align*}
for some constant $C_0>0$ which only depends on $\norm{\sigma}_4$. Therefore, $ \|\bK-\tilde\bK\|\le \|\bK-\tilde{\bK}\|_F\leq  C_0n\exp{(-B^2/4)}$ and
\begin{equation}\label{eq:tildeK_approx}
    \norm{\bK_\lambda^{-1/2}\left(\bK-\tilde\bK\right)\bK_\lambda^{-1/2}}\le \frac{C_0}{\lambda
    _0}n\exp{(-B^2/4)}.
\end{equation}
For \begin{align}\label{eq:CondB}
B\geq\sqrt{4\log \left(\frac{2C_0n}{\lambda_0}\right)} ,
\end{align}the above equation also implies that $\norm{\bK-\tilde\bK}\le \frac{\lambda_0}{2}$, and 
\begin{align}
    \norm{\tilde\bK_\lambda^{1/2}\bK_\lambda^{-1/2}}^2=&\norm{\bK_\lambda^{-1/2}\tilde\bK_\lambda\bK_\lambda^{-1/2}}\\
    \le& \norm{\bK_\lambda^{-1/2}\left(\bK-\tilde\bK\right)\bK_\lambda^{-1/2}}+\norm{\bK_\lambda^{-1/2}\bK_\lambda\bK_\lambda^{-1/2}}\le \frac{3}{2}.\label{eq:K_lambda-1_k}
\end{align}
Therefore, the smallest eigenvalues of $\bK_\lambda$ and $\tilde\bK_\lambda $ satisfy
\begin{equation}\label{eq:smallest_eig_tilde}
    \lambda_{\min}(\tilde\bK_\lambda )\ge\lambda_{\min}(\bK_\lambda )-\norm{\bK-\tilde\bK}\ge \frac{\lambda_0}{2}>0.
\end{equation}
It suffices to analyze $\norm{\bK_\lambda^{-1/2}(\tilde\bK_N-\bK)\bK_\lambda^{-1/2}}$ because of \eqref{eq:KNnotK} and the following equation:
\begin{align}
    \Parg{\norm{\bK_\lambda^{-1/2}(\bK_N-\bK)\bK_\lambda^{-1/2}}\ge t}\le~ & \Parg{\norm{\bK_\lambda^{-1/2}(\tilde\bK_N-\bK)\bK_\lambda^{-1/2}}\ge t}+\Parg{\bK_N\neq \tilde\bK_N}.\label{eq:bound_tildeKN-K_1}
\end{align}
Meanwhile, by \eqref{eq:tildeK_approx}, \eqref{eq:CondB}, and \eqref{eq:K_lambda-1_k}, we know that
\begin{align}
    \norm{\bK_\lambda^{-1/2}(\tilde\bK_N-\bK)\bK_\lambda^{-1/2}}\le~ & \norm{\bK_\lambda^{-1/2}(\tilde\bK_N-\tilde\bK)\bK_\lambda^{-1/2}}+\norm{\bK_\lambda^{-1/2}(\tilde\bK-\bK)\bK_\lambda^{-1/2}}\nonumber\\
    \le~ & \norm{\bK_\lambda^{-1/2}\tilde\bK_\lambda^{1/2}}^2\norm{\tilde\bK_\lambda^{-1/2}(\tilde\bK_N-\tilde\bK)\tilde\bK_\lambda^{-1/2}}+\norm{\bK_\lambda^{-1/2}(\tilde\bK-\bK)\bK_\lambda^{-1/2}}\nonumber\\
    \le ~& \frac{3}{2}\norm{\tilde\bK_\lambda^{-1/2}(\tilde\bK_N-\tilde\bK)\tilde\bK_\lambda^{-1/2}}+\frac{C_0 n}{\lambda_0}\exp{(-B^2/4)}\label{eq:bound_tildeKN-K}
\end{align}
Hence, we only need to prove the concentration inequality for $\tilde\bK_\lambda^{-1/2}(\tilde\bK_N-\tilde\bK)\tilde\bK_\lambda^{-1/2}=\sum_{i=1}^N \bH_i-\E\bH_i$. In terms of the definition of $\tilde{\sigma}$ and \eqref{eq:smallest_eig_tilde}, we know that, almost surely, 
\begin{align}\sup_{\w_i}\norm{\bH_i-\E\bH_i}\le 2\sup_{\w_i}\norm{\bH_i} &\le \sup_{\w_i}\frac{4}{\lambda_0 N }\norm{\tilde\sigma(\w_i^\top \bX)}^2 \\
&\le \frac{4n}{\lambda_0 N } \sup_{|x|\leq B} |\sigma(x)|^2 \\
&\leq \frac{4 C_{\sigma}^2 n  (1+B)^{2C_{\sigma}}}{\lambda_0 N },
\end{align}
where the last inequality is due to Assumption \ref{assump:sigma}.

 Analogously, applying \eqref{eq:smallest_eig_tilde}, we have 
\begin{align*}
    \bH_i^2&=\frac{1}{N^2}\tilde\bK_\lambda^{-1/2}\tilde\sigma(\w_i^\top \bX)^\top\tilde\sigma(\w_i^\top \bX)\tilde\bK_\lambda^{-1}\tilde\sigma(\w_i^\top \bX)^\top\tilde\sigma(\w_i^\top \bX)\tilde\bK_\lambda^{-1/2}\\
    & \preccurlyeq \frac{2\norm{\tilde\sigma(\w_i^\top \bX)}^2  }{\lambda_0 N^2}\tilde\bK_\lambda^{-1/2}\tilde\sigma(\w_i^\top \bX)^\top\tilde\sigma(\w_i^\top \bX)\tilde\bK_\lambda^{-1/2}\\
    & \preccurlyeq \frac{2 C_{\sigma}^2 n  (1+B)^{2C_{\sigma}}  }{\lambda_0 N^2}\tilde\bK_\lambda^{-1/2}\tilde\sigma(\w_i^\top \bX)^\top\tilde\sigma(\w_i^\top \bX)\tilde\bK_\lambda^{-1/2}.
\end{align*}
Notice that $\E[\tilde\sigma(\w_i^\top \bX)^\top\tilde\sigma(\w_i^\top \bX)]=\tilde\bK$. Hence, $\E[\tilde\bK_\lambda^{-1/2}\tilde\sigma(\w_i^\top \bX)^\top\tilde\sigma(\w_i^\top \bX)\tilde\bK_\lambda^{-1/2}]=\frac{1}{1+\lambda}\Id$, and
\[\E[\left(\bH_i-\E[\bH_i]\right)^2]\preccurlyeq \E \bH_i^2 \preccurlyeq \frac{2 C_{\sigma}^2 n  (1+B)^{2C_{\sigma}}  }{\lambda_0 N^2}\Id.\]
Thus, applying Theorem 5.4.1 of \cite{vershynin2018high}, we obtain
\begin{equation}\label{eq:KN_concentration1}
    \P \left(\norm{\sum_{i=1}^N\bH_i-\E\bH_i}>t\right)\le 2n \exp{\left(-\frac{t^2/2}{v+at/3}\right)},
\end{equation}
where \[v \le \frac{2 C_{\sigma}^2 n  (1+B)^{2C_{\sigma}}  }{\lambda_0 N} ,\quad a=\frac{4 C_{\sigma}^2 n  (1+B)^{2C_{\sigma}}}{\lambda_0 N }.\]

Take $t=C_{\sigma}^2(1+B)^{2C_{\sigma}}\sqrt{n/N}$ and $B=C'\sqrt{\log N}$. Then for $N\geq (1+B)^{4C_{\sigma}} n$, by taking constant $C'>0$ sufficiently large, \eqref{eq:CondB} holds and the right hand side of \eqref{eq:KNnotK} is no great than $\frac{1}{2}N^{-2}$. Moreover, \eqref{eq:KN_concentration1} implies that there exists an absolute constant $C''>0$ such that 
\begin{equation}\label{eq:KN_concentration2}
    \P \left(\norm{\tilde\bK_\lambda^{-1/2}(\tilde\bK_N-\tilde\bK)\tilde\bK_\lambda^{-1/2}}>C''\log^{C_{\sigma}}(N)\sqrt{\frac{n}{N}}\right)\le \frac{1}{2} N^{-2},
\end{equation} for sufficiently large $C'$. Here both $C',C''>0$ are determined by $\lambda_0$ and $C_{\sigma}$. Notice that for all large $N$, the second term of \eqref{eq:bound_tildeKN-K} can be also bounded by $C''' \log^{C_{\sigma}} (N)\sqrt{\frac{n}{N}}$ for some constant $C'''>0$ only depending on $C_{\sigma}$ and $\lambda_0$.

Combining \eqref{eq:KNnotK}, \eqref{eq:bound_tildeKN-K_1}, \eqref{eq:bound_tildeKN-K}, and \eqref{eq:KN_concentration2}, we can conclude that there exists some large constants $C_1, C_2>0$, such that  with probability at least $1-N^{-2}$, when $N/\log^{2C_{\sigma}}(N)> C_1n$, and $n\geq n_0$,
\begin{align}
    \norm{\bK_\lambda^{-1/2}\left(\bK_N-\bK\right)\bK_\lambda^{-1/2} } 
    \leq & C_2 \log^{C_{\sigma}}(N)\sqrt{\frac{n}{N}}, 
\end{align}
as desired, where $C_1, C_2$ are constants depending only on $C_{\sigma}$ and $\lambda_0$.
\end{proof}

\subsection{Proof of Theorem~\ref{thm:train_diff}}

We first prove the following corollary from Proposition \ref{prop:kernel_con}.
\begin{cor}\label{coro_con}
Following the notations of Proposition~\ref{prop:kernel_con}, let us denote $t=C_1\log^{C_{\sigma}}(N)\sqrt{\frac{n}{N}}$. When $t\in (0,1)$, under the same assumptions as Proposition~\ref{prop:kernel_con}, with probability at least $1-N^{-2}$, when $n\geq n_0$, the following holds:
\begin{align}
      \norm{(\bK_N+\lambda\Id)^{-1/2}\left(\bK-\bK_N\right)(\bK_N+\lambda\Id)^{-1/2}}\le &t, \notag\\
      \norm{\bK^{-1/2}_\lambda(\bK_N+\lambda\Id)^{1/2}}\le &\sqrt{1+t},\notag\\
      \norm{\bK^{1/2}_\lambda(\bK_N+\lambda\Id)^{-1/2}}\le &(1-t)^{-1/2}, \label{eq:KlambdaKN}
\end{align}
and the smallest eigenvalue $\lambda_{\min}(\bK_N)\ge (1-t)\lambda_0$. 
\end{cor}
\begin{proof}
Based on Proposition~\ref{prop:kernel_con}, under the  event in \eqref{eq:kernel_con}, we can deduce that 
\begin{align}
      (\bK_N-\bK)  \preccurlyeq~& t\bK_\lambda, \notag\\
     \left(\bK-\bK_N\right)  \preccurlyeq~& t\bK_\lambda, \notag\\
     \left(\bK-\bK_N\right)  \preccurlyeq ~&\frac{t}{1-t} (\bK_N+\lambda\Id), \notag \\
   (\bK_N+\lambda\Id)  \preccurlyeq~&(1+t)\bK_\lambda, \notag \\
     \bK_\lambda\preccurlyeq ~&\frac{1}{1-t}(\bK_N+\lambda\Id), \label{eq:Klambda_dominance}
\end{align}with probability at least $1-N^{-2}$. These imply the results of the Corollary~\ref{coro_con}, where the bound of $\lambda_{\min}(\bK_N)$ is due to Proposition~\ref{prop:pkrr} and \eqref{eq:Klambda_dominance}.
\end{proof}

Now we are ready to prove Theorem~\ref{thm:train_diff}.
\begin{proof}[Proof of Theorem~\ref{thm:train_diff}]
From the definitions of training errors in \eqref{eq:Etrain_Kn} and \eqref{eq:Etrain_K}, Proposition~\ref{prop:pkrr} and Corollary~\ref{coro_con} implies  that both $\bK(\bX,\bX)$ and $\bK_N(\bX,\bX)$ are invertible with probability at least $1-N^{-2}$ when $t\in (0,3/4)$.
Thus, we have when $t\in (0,3/4)$, with probability at least $1-N^{-2}$, when $n\geq n_0$,
\begin{align}
     &\left| E_{\train}^{(\RF,\lambda)}-E_{\train}^{(\KK,\lambda)}\right|= \frac{\lambda^2}{n} \left|\Tr[(\bK+\lambda\Id)^{-2} \v y\v y^\top]-\Tr[(\bK_N +\lambda\Id)^{-2} \v y\v y^\top] \right| \notag \\
     =~& \frac{\lambda^2}{n} \left|\v y^\top\left[(\bK +\lambda\Id)^{-2}-(\bK_N +\lambda\Id)^{-2}\right] \v y\right| \notag \\
     \leq ~&\frac{\lambda^2}{n} \|(\bK +\lambda\Id)^{-2}-(\bK_N +\lambda\Id)^{-2}\|\cdot  \|\v y\|^2 \notag\\
     \leq ~& \frac{\lambda^2\|\v y\|^2}{n} \|(\bK +\lambda\Id)^{-1}-(\bK_N +\lambda\Id)^{-1}\|\cdot(\|(\bK +\lambda\Id)^{-1}\|+\|(\bK_N +\lambda\Id)^{-1}\|) \notag\\
     \leq ~&\frac{5\lambda^2\|\v y\|^2}{\lambda_0 n} \| (\bK +\lambda\Id)^{-1}-(\bK_N +\lambda\Id)^{-1} \|, \label{eq:last_thm37}
\end{align}
where in the last line, we employ the fact that
$\norm{(\bK_N +\lambda\Id)^{-1}}\leq  4\lambda_0^{-1}$ and $\norm{(\bK +\lambda\Id)^{-1}} \leq \lambda_0^{-1}$ from Corollary \ref{coro_con} and Proposition \ref{prop:pkrr}, respectively.

Take $C_2=\sqrt{2} C_1$ in Proposition~\ref{prop:kernel_con}. For any $N$  satisfying $N/\log^{2C_{\sigma}}(N)> 2C_1^2 n$, we can make $0<t< 3/4$, where $t=C_1\log^{C_{\sigma}}(N)^{C_{\sigma}}\sqrt{\frac{n}{N}}$. From this, considering the identity $A^{-1}-B^{-1}=B^{-1}(B-A)A^{-1}$, and applying Proposition~\ref{prop:kernel_con} and Corollary~\ref{coro_con}, we obtain that
\begin{align}
    & \norm{ (\bK +\lambda\Id)^{-1}-(\bK_N +\lambda\Id)^{-1} }\\
    =& \norm{(\bK_{\lambda})^{-1/2}(\bK_{\lambda})^{-1/2}\left(\bK_N -\bK \right)(\bK_{\lambda})^{-1/2}(\bK_{\lambda})^{1/2}(\bK_N +\lambda\Id)^{-1/2}(\bK_N +\lambda\Id)^{-1/2}}\\
     \le&\frac{1}{2\lambda_0}\norm{ (\bK_{\lambda})^{-1/2}\left(\bK_N -\bK \right)(\bK_{\lambda})^{-1/2}} \norm{(\bK_{\lambda})^{1/2}(\bK_N +\lambda\Id)^{-1/2}}\\
     \le &\frac{t}{2\lambda_0 \sqrt{1-t}}\le \frac{t}{\lambda_0}.\label{eq:inv_KK}
\end{align}
Hence, from \eqref{eq:last_thm37}, we get
\[ \left| E_{\train}^{(\RF,\lambda)}-E_{\train}^{(\KK,\lambda)}\right|\leq \frac{5\lambda^2\norm{\by}^2t}{\lambda_0^2 n},\]
which finishes the proof of Theorem~\ref{thm:train_diff}.
\end{proof}
 
\subsection{Proof of Theorem~\ref{thm:cv_diff}}\label{sec:LOOCV}

We start with the following estimate on the \textit{normalized} trace $\tr \bK_{\lambda}^{-1}$.
\begin{lemma}\label{lemm:K_lambda_tr}
Under Assumption~\ref{assump:sigma}, we have $\tr \bK_\lambda=\lambda+\norm{\sigma}_2^2$ and when $n\geq n_0$,
\[(\lambda+\norm{\sigma}_2^2)^{-1}\le \tr\bK_\lambda^{-1}\le \lambda_0^{-1}.\]
\end{lemma}
\begin{proof}
By definition of $\bK$, we know $\Tr[\bK]=n\E_{\bw}[\sigma(\bw^\top\bx)^2]=n\E[\sigma(\xi)^2]=n\norm{\sigma}_2^2$ for $\xi\sim\cN(0,1)$. Hence, $\Tr \bK_\lambda=n\left(\lambda+\norm{\sigma}_2^2\right)$. Denote $\lambda_1\ge \lambda_2\ge\cdots\ge \lambda_n\ge 0$ as the eigenvalues of $\bK$. Then, by Cauchy–Schwartz inequality, we have
\begin{align*}
    n=\sum_{i=1}^n \frac{1}{\sqrt{\lambda_i+\lambda}}\sqrt{\lambda_i+\lambda}\le \left(\Tr\bK_\lambda^{-1}\right)^{1/2}\left(\Tr\bK_\lambda\right)^{1/2}.
\end{align*}
Therefore, we can get $(\lambda+\norm{\sigma}_2^2)^{-1}\le \tr\bK_\lambda^{-1}$. Meanwhile, based on Proposition~\ref{prop:pkrr}, $\tr\bK_\lambda^{-1}\le \norm{\bK_\lambda^{-1}}\le \lambda^{-1}_{\min}(\bK)\le\lambda_0^{-1}$. Notice that $\lambda_0=\frac{1}{2}\sigma_{>\ell}^2\le \norm{\sigma}_2^2$.
\end{proof}

Recall that $\bK_{N,\lambda}=\bK_N+\lambda\Id$. The following lemma for the approximation of $\bK_{\lambda}$ with $\bK_{N,\lambda}$ holds.
\begin{lemma}\label{lemm:trK_lambda-1}
Under the assumptions of Proposition~\ref{prop:kernel_con}, for sufficiently large constant $C>0$, when $N/\log^{2C_{\sigma}}(N)> C(1+\lambda^2) n$, we have, with probability at least $1-N^{-2}$, when $n\geq n_0$,
\begin{equation}\label{eq:K_l-1-KNl-1}
    \left|\tr\bK_\lambda^{-1}-\tr\bK_{N,\lambda}^{-1}\right|\le C_0\log^{C_{\sigma}}(N)\sqrt{\frac{n}{N}},
\end{equation}
and 
\begin{equation}
    \frac{1}{2(\lambda+\norm{\sigma}_2^2)}\le \tr\bK_{N,\lambda}^{-1}\le\frac{3}{2\lambda_0},
\end{equation}
where constants $C, C_0>0$ only depends on $\lambda_0$ and $\norm{\sigma}_4$.
\end{lemma}
\begin{proof}
From \eqref{eq:inv_KK} and Lemma~\ref{lemm:K_lambda_tr}, by taking $t=C_1\log^{C_{\sigma}}(N)\sqrt{\frac{n}{N}}\in (0,1)$, we have
\begin{align}
    \left|\tr\bK_\lambda^{-1}-\tr\bK_{N,\lambda}^{-1}\right|\leq \norm{\bK_{\lambda}^{-1}-\bK_{N,\lambda}^{-1}}\leq \frac{t}{\lambda_0 (1-t)},\\
   (\lambda+\norm{\sigma}_2^2)^{-1}- \frac{t}{\lambda_0(1-t)}\le \tr\bK_{N,\lambda}^{-1}\le \lambda_0^{-1}+\frac{t}{\lambda_0(1-t)}.\label{eq:K_l-1-KNl-12}
\end{align}
Considering sufficiently large constant $C>0$ with $N/\log^{2C_{\sigma}}(N)> C(1+\lambda^2) n$, we can ensure that $t$ is sufficiently small and satisfies $0\le t \le \min\{1/2,\lambda_0/4(\lambda+\norm{\sigma}_2^2)\}$. Then,
\begin{equation}
    \label{eq:K_l-1-KNl-13}
    \frac{t}{\lambda_0(1-t)}\le \frac{1}{2(\lambda+\norm{\sigma}_2^2)}\le \frac{1}{2\lambda_0}.
\end{equation}
Hence, taking $C_0=2C_1/\lambda_0$, we can conclude \eqref{eq:K_l-1-KNl-1}. The second statement follows from \eqref{eq:K_l-1-KNl-12} and \eqref{eq:K_l-1-KNl-13} directly.
\end{proof}

\begin{lemma}\label{lemm:shortcut}
Based on the definitions of LOOCVs of KRR and RFRR in \eqref{eq:CV}, we have shortcut formulae \eqref{eq:shortcut} and \eqref{eq:shortcut_N} for KRR and RFRR respectively: for any $\lambda\ge 0$,
\begin{align}
     \CV_n^{(\KK,\lambda)}=~ & \frac{1}{n}\by^\top \bK_{\lambda}^{-1}\bD^{-2} \bK_{\lambda}^{-1}\by, \\
      CV_n^{(\RF,\lambda)}= ~& \frac{1}{n}\by^\top \bK_{N,\lambda}^{-1}\bD_N^{-2} \bK_{N,\lambda}^{-1}\by,
\end{align}
where $\bD$ and $\bD_N$ are diagonal matrices with diagonals $[\bD]_{ii}=[\bK_{\lambda}^{-1}]_{ii}$ and $[\bD_N]_{ii}=[\bK_{N,\lambda}^{-1}]_{ii}$, for $i\in[n]$, respectively. When $n\geq n_0$, we have 
\begin{equation}\label{eq:K_lambda_ii}
   (\lambda+\norm{\sigma}_2^2)^{-1}\le \norm{\bD}\le \lambda_0^{-1}.
\end{equation}
Additionally, for a constant $C>0$ depending on $\lambda_0,\|\sigma\|_2$, when $N/\log^{2C_{\sigma}}(N)>C(1+\lambda^2)n$,
\begin{align}
    \frac{1}{2}(\lambda+\norm{\sigma}_2^2)^{-1}\le \norm{\bD_N}\le 2\lambda_0^{-1},\label{eq:K_lambdaN_ii}\\
    \norm{\bD^{-2}-\bD_N^{-2}}\le  C_0(1+\lambda^{4})\log^{C_{\sigma}}(N)\sqrt{\frac{n}{N}},\label{eq:D-D_N}
\end{align}
with probability at least $1-N^{-2}$, for some constant $C_0>0$ which only depends on $\lambda_0$ and $\norm{\sigma}_2$.
\end{lemma}
\begin{proof}
For $i\in[n]$, denote $\by^{-i}\in\R^{n-1}$ by the vector $\by$ with the $i$-th entry removed, $\bX^{-i}$ by the data $\bX$ with the $i$-th colmun removed, and $\bK_{-i,\lambda}$ by the matrix $\bK_\lambda$ with both the $i$-th row and column removed. Based on Schur complement and resolvent identities \cite[Lemma 3.5]{benaych2016lectures}, we have for any $i,j\in [n]$ with $j\neq i$, the $(i,j)$ entry of $\bK_\lambda^{-1}$ is given by
\begin{align}\label{eq:resolvent_id}
    [\bK_\lambda^{-1}]_{i,j}=-[\bK_{\lambda}^{-1}]_{ii}\sum_{k\neq i} [\bK]_{i,k}[\bK_{-i,\lambda}^{-1}]_{k,j}.
\end{align}
Thus, from definition \eqref{eq:CV}, we can exploit \eqref{eq:resolvent_id} to obtain 
\begin{align}
    & y_i-\hat{f}_{\lambda,-i}^{(\KK)}(\bx_i)=y_i-\bK(\bx_i,\bX^{-i})\bK_{-i,\lambda}^{-1}\by^{-i}\\
    =~& y_i+ \frac{[\bK_\lambda^{-1}]_{[i,\neq i]}\by^{-i}}{[\bK_\lambda^{-1}]_{ii}}+\left(\frac{[\bK_\lambda^{-1}]_{ii}}{[\bK_\lambda^{-1}]_{ii}} y_i-y_i\right)
    = \frac{[\bK_\lambda^{-1}]_{[i,:]}\by}{[\bK_\lambda^{-1}]_{ii}},
\end{align}
for any $i\in [n]$, where $[\bK_\lambda^{-1}]_{[i,\neq i]}$ is the $i$-th row of $\bK_\lambda^{-1}$ with the $i$-th entry removed, and $[\bK_\lambda^{-1}]_{[i,:]}$ is the $i$-th row of $\bK_\lambda^{-1}$. Hence, in matrix form, we can get 
\begin{align}
    \CV_n^{(\KK,\lambda)}=\frac{1}{n}\sum_{i=1}^n\frac{\by^\top[\bK_\lambda^{-1}]_{[i,:]}^\top [\bK_\lambda^{-1}]_{[i,:]}\by}{[\bK_\lambda^{-1}]_{ii}}=\frac{1}{n}\by^\top \bK_{\lambda}^{-1}\bD^{-2} \bK_{\lambda}^{-1}\by.
\end{align}
Going through the same procedure, we can verify \eqref{eq:shortcut_N} as well.

Secondly, applying Theorem A.4 of \cite{bai2010spectral}, we have
\begin{equation}
   [\bD]_{ii}= [\bK_{\lambda}^{-1}]_{ii} = \frac{1}{[\bK_{\lambda}]_{ii}-\bK(\bx_i,\bX^{-i})\bK_{-i,\lambda}^{-1}\bK(\bX^{-i},\bx_i)}, \notag
\end{equation} for any $i\in [n]$. Recall that, in the proof of Lemma~\ref{lemm:K_lambda_tr}, we have shown $[\bK_{\lambda}]_{ii}=\lambda+\norm{\sigma}^2_2$ for all $i\in [n]$. Therefore, we have 
\begin{equation}
    (\lambda+\norm{\sigma}_2^2)^{-1}\le [\bK_{\lambda}^{-1}]_{ii}\le  \| \bK_{\lambda}^{-1}\|\leq \lambda_0^{-1},~\forall i\in [n].
\end{equation}
which verifies the result \eqref{eq:K_lambda_ii}.

Meanwhile, from the proof of Lemma~\ref{lemm:trK_lambda-1}, for sufficiently large constants $C,C_1$ depending only on $\lambda_0, \|\sigma\|_2$, with  $t=C_1\log^{C_{\sigma}}(N)\sqrt{n/N}\leq \min \{ 1/2, \lambda_0/4(\lambda+\norm{\sigma}_2^2)\} $ and $N/\log^{2C_{\sigma}}(N)>C(1+\lambda^2) n$, we have 
\begin{equation}\label{eq:D-D_N1}
    \norm{\bD-\bD_N}\le \norm{\bK_{\lambda}^{-1}-\bK_{N,\lambda}^{-1}}\leq \frac{t}{\lambda_0 (1-t)},
\end{equation}
with probability at least $1-N^{-2}$. Therefore, we can verify \eqref{eq:K_lambdaN_ii} as follows
\begin{align}
    \frac{1}{2}(\lambda+\norm{\sigma}_2^2)^{-1}\le(\lambda+\norm{\sigma}_2^2)^{-1}-\norm{\bD-\bD_N}\le \norm{\bD_N}\le \lambda_0^{-1}+\norm{\bD-\bD_N}\le 2\lambda_0^{-1}.
\end{align}

Finally, combining \eqref{eq:K_lambda_ii}, \eqref{eq:K_lambdaN_ii} and \eqref{eq:D-D_N1} together, we can obtain that 
\begin{align}
     \norm{\bD^{-2}-\bD^{-2}_N}\le & \norm{\bD}^{-2}\norm{\bD_N}^{-2}\left(\norm{\bD_N}+\norm{\bD}\right)\norm{\bD-\bD_N}\\
     \le &~12\frac{(\lambda+\norm{\sigma}_2^2)^{4}}{\lambda_0^3} t\le C_0(1+\lambda^4)\log^{C_{\sigma}}(N)\sqrt{\frac{n}{N}}.
\end{align}This finally completes the proof of this lemma.
\end{proof}

\begin{proof}[Proof of Theorem~\ref{thm:cv_diff}]
We start with \eqref{eq:GCV_K}. Recall \eqref{eq:Etrain_Kn}, \eqref{eq:Etrain_K}, and $\bK_{N,\lambda}=\bK_N+\lambda\Id$. Using the expression \eqref{eq:GCV}, we have
\begin{align}
     |\GCV_n^{(\KK,\lambda)}-\GCV_n^{(\RF,\lambda)}|\leq &\frac{1}{\lambda^2}\left|\left(\left(\tr \bK_{\lambda}^{-1}\right)^{-2}-\left(\tr \bK_{N,\lambda}^{-1}\right)^{-2}\right) E_{\train}^{(\KK,\lambda)}\right|
      +\frac{1}{\lambda^2\left(\tr \bK_{\lambda}^{-1}\right)^{2}}\left| E_{\train}^{(\KK,\lambda)}-E_{\train}^{(\RF,\lambda)}\right|\\
      \le &  \frac{1}{n}\norm{  \bK_\lambda^{-1}\v y}^2\left|\left(\tr \bK_{\lambda}^{-1}\right)^{-2}-\left(\tr \bK_{N,\lambda}^{-1}\right)^{-2}\right|\label{eq:GCV_1}\\
      &+ C_2(\lambda +\|\sigma\|_2^2)^2 \frac{\log^{C_{\sigma}}(N) \|\by\|^2}{\sqrt{nN}},\label{eq:GCV_2}
\end{align}
where \eqref{eq:GCV_2} is due to  \eqref{eq:ETrain} and Lemma \ref{lemm:K_lambda_tr}, and $C_2$ is a constant depending on $\|\sigma\|_4$ and $\lambda_0$.

  For the first term \eqref{eq:GCV_1}, when $N/\log^{2C_{\sigma}} N\geq C(1+\lambda^2) n$ for a sufficiently large $C$, together with Lemmas~\ref{lemm:K_lambda_tr} and \ref{lemm:trK_lambda-1}, we can show that with probability at least $1-N^{-2}$,
\begin{align*}
    &\frac{1}{n}\norm{  \bK_\lambda^{-1}\v y}^2\left|\left(\tr \bK_{\lambda}^{-1}\right)^{-2}-\left(\tr \bK_{N,\lambda}^{-1}\right)^{-2}\right|\\
    \leq~ &\left(\tr\bK_{\lambda}^{-1}\right)^{-2}\left(\tr\bK_{N,\lambda}^{-1}\right)^{-2} \left|\tr (\bK_{\lambda}^{-1}-\bK_{N,\lambda}^{-1})\right|\tr \left(\bK_{\lambda}^{-1}+\bK_{N,\lambda}^{-1}\right)\frac{1}{n}\|\bK_{\lambda}^{-2}\| \norm{\by}^2\\
    \leq~ &   \frac{20(\lambda+\norm{\sigma}_2^2)^{4} }{\lambda_0^3 n}\norm{\by}^2C_0\log^{C_{\sigma}}(N)\sqrt{\frac{n}{N}}
    \leq \frac{C_1(1+\lambda^4)\norm {\by}^2\log^{C_{\sigma}}(N)}{\sqrt{nN}},
\end{align*}for some constant $C_1$ which only relies on $\norm{\sigma}_2$ and $\lambda_0$.
Hence, the bounds of \eqref{eq:GCV_1} and \eqref{eq:GCV_2} imply 
\begin{align}
     |\GCV_n^{(\KK,\lambda)}-\GCV_n^{(\RF,\lambda)}|\leq  \frac{ c(1+\lambda^4)\log^{C_{\sigma}}(N)\norm{\by}^2}{\sqrt{nN}} 
\end{align}
for a constant $c$ depending on $\lambda_0$, $\norm{\sigma}_2$, and $\norm{\sigma}_4$. This proves \eqref{eq:GCV_K} for the  GCV concentration.

Now we consider the second result \eqref{eq:CV_K} for LOOCV. 
Similar to the analysis of GCV, with the help of the shortcut formulae \eqref{eq:shortcut} and \eqref{eq:shortcut_N}, we can get
\begin{align}
    \left| \CV_n^{(\KK,\lambda)}- \CV_n^{(\RF,\lambda)}\right|\le~& \frac{\|\by\|^2}{n}\norm{(\bK_{\lambda}^{-1}-\bK_{N,\lambda}^{-1})\bD^{-2}\bK_\lambda^{-1}}+ \frac{\|\by\|^2}{n} \norm{\bK^{-1}_{N,\lambda}(\bD^{-2}-\bD^{-2}_N)\bK_\lambda^{-1}}\\
    &+ \frac{\|\by\|^2}{n} \norm{\bK^{-1}_{N,\lambda} \bD^{-2}_N(\bK_{\lambda}^{-1}-\bK_{N,\lambda}^{-1})}
    \le  \frac{ c(1+\lambda^4)\log^{C_{\sigma}}(N)\norm{\by}^2}{\sqrt{nN}},\label{eq:CV_derive}
\end{align}
with probability at least $1-N^{-2}$. Here, we exploit \eqref{eq:K_lambda_ii}, \eqref{eq:K_lambdaN_ii} and \eqref{eq:D-D_N} in Lemma~\ref{lemm:shortcut}. This verifies the second result \eqref{eq:CV_K} for LOOCV.
\end{proof}

\subsection{Proof of Theorem~\ref{thm:test_diff}}
For simplicity, we denote $\KNl=(\bK_N+\lambda\Id)$, $\Kl=(\bK+\lambda\Id)$, $\KmN:=\bK_N(\bX,\bx)\in\R^n$ and $\Km:=\bK(\bX,\bx)\in\R^n$. Define
\[\KxxN:=\E_\bx[\bK_N(\bX,\bx)\bK_N(\bx,\bX)],~~\Kxx:=\E_\bx[\bK(\bX,\bx)\bK(\bx,\bX)],\]
where we take expectation with respect to test data $\bx$ defined in Assumption~\ref{assump:testdata}.  Recalling Assumption \ref{assump:target}, we have
\[f^*(\bx)=\tau(\langle\bbeta,\bx\rangle),\quad \by=\tau(\bX^\top\bbeta)+\beps,\]
where $\bbeta\sim\cN(0,\Id)$. Denote the kernel given by $\tau$ as $\bPsi:=\E_\bbeta[\tau(\bX^\top\bbeta)\tau(\bbeta^{\top}\bX)]$ and the vector by $\bu:=\E_\bbeta[\tau(\bX^\top\bbeta) f^*(\bx)]\in\R^n$.

We begin with the following lemmas about the bounds and concentrations with respect to $\KxxN$, $\Kxx$, $\bK_{N,\lambda} $,  $\bK_{m,N}$, and $\bK_{m}$.

\begin{lemma}\label{lemm:Kmix_bound}
There exist some constant $C>0$ depending on $\lambda_0, \|\sigma\|_2^2$ such that, with probability at least $1-N^{-1}$, when $n\geq \max \{n_0,n_1\}$,
\[\KmN^\top\KNl^{-1}\KmN< \norm{\sigma}_2^2+ \|\sigma\|_4^2+\lambda,\] when $N/\log^{2C_{\sigma}}(N)>C(1+\lambda^2)n$.
Moreover, we have $$\Km^\top\Kl^{-1}\Km<\norm{\sigma}_2^2+\lambda.$$
\end{lemma}
\begin{proof}
Consider an enlarged block matrix $\tilde\bK\in\R^{(n+1)\times (n+1)}$ defined by
\begin{equation}\label{eq:tilde_K}
    \tilde\bK:=\begin{pmatrix}
\bK & \Km\\
\Km^\top & \bK(\bx,\bx)
\end{pmatrix},
\end{equation}
where $\bK(\bx,\bx)=\E[\sigma(\w^\top\bx)(\sigma(\w^\top \bx)]=\norm{\sigma}_2^2$. Let $\tKl:=\tilde\bK+\lambda\Id $. Analogously to \eqref{def:Phi0}, let us define
\begin{align}
     \tilde\bK_\ell:=\sum_{k=0}^\ell \zeta_k^2(\sigma)(\tilde\bX^\top \tilde\bX)^{\odot k}+\sigma_{>\ell}^2\Id\in\R^{(n+1)\times (n+1)},\label{eq:tilde_K_l}
 \end{align}
where $\tilde\bX=[\bX,\bx]\in\R^{d\times (n+1)}$ is the concatenation of training and test data points. By Assumption~\ref{assump:testdata}, analogously to the proof of  Proposition~\ref{prop:pkrr}, we have 
\begin{align}\label{eq:tilde_K-tilde_K_l}
    \norm{\tilde\bK_\ell-\tilde\bK}^2 &\le  \norm{\tilde\bK_\ell-\tilde\bK}_F^2\\
    &\le 2\|\sigma\|_4^4 \left\| \left(\tilde \bX^\top  \tilde \bX\right)^{\odot \ell+1}-\Id\right\|_F^2\\
    &=2\|\sigma\|_4^4 \left\| \left( \bX^\top   \bX\right)^{\odot \ell+1}-\Id\right\|_F^2+ 4\|\sigma\|_4^4\left\| (\bX^\top \bx)^{\odot (\ell+1)}\right\|_2^2 \leq \frac{3}{8} \lambda_0^2.
\end{align}
 Since $\lambda_{\min}(\tilde\bK)\ge \lambda_0>0$, we have $\lambda_{\min}(\tKl)\geq \frac{1}{4}\lambda_0$ and $\tKl$ is positive definite for any $\lambda\ge 0$. By Theorem 7.7.7 of \cite{horn2012matrix}, since both $\bK_{\lambda}$ and $\tilde{\bK}_{\lambda}$ are positive definite, the Schur complement of $\tilde{\bK}_{\lambda}$ given by $\bK(\bx,\bx)+\lambda-\Km^\top\Kl^{-1}\Km$ is positive, which concludes our second result in this lemma.

Similarly, consider the block matrix $\tilde\bK_N\in\R^{(n+1)\times (n+1)}$ defined by 
\[\tilde\bK_N:=\bK_N(\tilde\bX,\tilde\bX)=\begin{pmatrix}
 \bK_N & \KmN\\
\KmN^\top & \bK_N(\bx,\bx) 
\end{pmatrix}.\]
Let $\tKNl:=\lambda\Id+ \tilde\bK_N$. Combining Assumption \ref{assump:testdata} and \eqref{eq:tilde_K-tilde_K_l}, we can easily ensure Proposition~\ref{prop:kernel_con} and Corollary~\ref{coro_con} still hold for $\tKl$ and $\tKNl$. Therefore, with probability at least $1-N^{-2}$, $\lambda_{\min}(\tKNl)\ge \frac{1}{2}\lambda_0$, for sufficiently large constant $C>0$ with $N/\log^{2C_{\sigma}}(N)>C(1+\lambda^2) n$, which implies that $\tKNl $ is positive definite with  probability $1-N^{-2}$. Again, from Theorem 7.7.7 of \cite{horn2012matrix}, we can get $\bK_N(\bx,\bx)+\lambda-\KmN^\top\KNl^{-1}\KmN>0$ with probability at least $1-N^{-2}$. Thus, 
\begin{align*}
    0\le \KmN^\top\KNl^{-1}\KmN<\bK_N(\bx,\bx)+\lambda.
\end{align*}
Notice that $\bK_N(\bx,\bx)=\frac{1}{N}\sum_{i=1}^N\sigma(\bw_i^\top\bx)^2$, $\E[\bK_N(\bx,\bx)]=\norm{\sigma}_2^2,$ and 
\begin{align}
    \E\left(\bK_N(\bx,\bx)-\norm{\sigma}_2^2\right)^2=\frac{1}{N}\Var (\sigma(\bw^\top\bx)^2)=\frac{\norm{\sigma}_4^4-\norm{\sigma}_2^4}{N}\le\frac{\norm{\sigma}_4^4}{N},
\end{align}where $\bw\sim\cN(0,\Id)$. Therefore, by Markov inequality, we conclude that, with probability at least $1-1/N$, $\bK_N(\bx,\bx)\le \norm{\sigma}_2^2+ \|\sigma\|_4^2$.
Therefore, with probability at least $1-N^{-1}$, we have $\KmN^\top\KNl^{-1}\KmN<\norm{\sigma}_2^2+ \|\sigma\|_4^2+\lambda.$
\end{proof}
\begin{lemma}\label{lemm:KPsiK}
Suppose that, for any $0\le k\le \ell,$ if $\zeta_k(\sigma)\neq 0$, then $\zeta_k(\tau)\neq 0$. Then, there exists a universal constant $C>0$ only depending on $\sigma$ and $\tau$ such that for any $\lambda\ge 0$ and $n\geq n_0$,
\[\norm{\Kl^{-1/2}\bPsi\Kl^{-1/2}}\le C,\]

\end{lemma}
\begin{proof}
Analogously to \eqref{def:Phi0}, we define a truncation version of the kernel $\bPsi$ by
\begin{align}
     \bPsi_\ell:=\sum_{k=0}^\ell \zeta_k^2(\tau)\left(\bX^\top \bX\right)^{\odot k}+\tau_{>\ell}^2\Id. \label{eq:Psi_ell}
 \end{align}
 Define $\bK_{\ell,\lambda}:=\bK_{\ell}+\lambda\Id$. By the assumption and definition of $\bK_\ell$, there exists some constant $C>0$ such that $\bPsi_\ell\preccurlyeq C\bK_{\ell,\lambda}$ for any $\lambda\ge 0$. Here this constant $C$ only relies on the Hermite coefficients $\zeta_k(\tau)$, $\lambda_0$ and $\zeta_k(\sigma)$ for $0\le k\le \ell$.
 Next, applying Proposition~\ref{prop:pkrr} for nonlinear function $\tau$, we have
 \begin{equation}\label{eq:Phil_approx}
     \norm{\bPsi-\bPsi_\ell}\le \sqrt{2} \|\tau\|_4^2 \left\| \left(\bX^\top \bX\right)^{\odot \ell+1}-\Id\right\|_F\leq \frac{\sqrt 2 \sigma_{>\ell}^2 \|\tau\|_4^2}{4 \|\sigma\|_4^2}.
 \end{equation}
 Proposition~\ref{prop:pkrr} also indicates that  $\norm{\bK-\bK_\ell}\le \frac{1}{2}\lambda_0$. This implies that 
 \begin{equation}
     \Kl^{-1/2}\bK_{\ell,\lambda} \Kl^{-1/2}\preccurlyeq \frac{3}{2}\Id,
 \end{equation} for any $\lambda\ge 0$.
 Then, for any $\lambda\ge 0$, we can estimate its contribution by
 \begin{align*}
     \norm{\Kl^{-1/2}\bPsi\Kl^{-1/2}}\le~& \norm{\Kl^{-1/2}(\bPsi-\bPsi_\ell)\Kl^{-1/2}}+\norm{\Kl^{-1/2}\bPsi_\ell\Kl^{-1/2}}\\
     \le~& \lambda_0^{-1}\norm{\bPsi-\bPsi_\ell}+\norm{\Kl^{-1/2}\bK_{\ell,\lambda}^{1/2}\bK_{\ell,\lambda}^{-1/2}\bPsi_\ell\bK_{\ell,\lambda}^{-1/2}\bK_{\ell,\lambda}^{1/2}\Kl^{-1/2}}\\
     \le~ & \frac{\sqrt 2 \sigma_{>\ell}^2 \|\tau\|_4^2}{4\lambda_0 \|\sigma\|_4^2}+ \norm{\bK_{\ell,\lambda}^{1/2}\Kl^{-1/2}}^2\norm{\bK_{\ell,\lambda}^{-1/2}\bPsi_\ell\bK_{\ell,\lambda}^{-1/2}}\\
     \le~ & \frac{\sqrt 2 \sigma_{>\ell}^2 \|\tau\|_4^2}{4\lambda_0 \|\sigma\|_4^2}+\frac{3}{2}C.
 \end{align*}
 Therefore,  there is a constant depending only on $\sigma,\tau$ as the upper bound for $ \norm{\Kl^{-1/2}\bPsi\Kl^{-1/2}}$. This completes the proof of this lemma.
\end{proof}

\begin{lemma}\label{lemm:KPsiK_vector}
There exists a constant $C>0$ depending only on $\sigma,\tau$ such that for any $\lambda\ge 0$ and  $n\geq \max \{n_0,n_1\}$,
\begin{equation}
    \norm{\bK_\lambda^{-\frac{1}{2}}\bu}=\norm{\E_{\bbeta}[\tau(\bbeta^\top\bx)\tau(\bbeta^\top\bX)]\bK_\lambda^{-\frac{1}{2}}}\le C(1+\lambda^{1/2}),
\end{equation}
\end{lemma}

\begin{proof}
 
Denote $\bK_{\tau,m}:=\E_{\bbeta}[\tau(\bbeta^\top\bx)\tau(\bbeta^\top\bX)]^\top$ and $\bPsi_{\lambda}:=\bPsi+\lambda\Id$. Analogously to \eqref{eq:tilde_K} and \eqref{eq:tilde_K_l}, we can consider
\begin{equation}
\tilde\bPsi_{\lambda}=\E_\bbeta[\tau(\bbeta^\top\tilde\bX)^\top\tau(\bbeta^\top\tilde\bX)]+\lambda\Id=\begin{pmatrix}
    \bPsi_{\lambda} & \bK_{\tau,m}\\
    \bK_{\tau,m}^\top & \E_{\bbeta}[\tau(\bbeta^\top\bx)^2]+\lambda
    \end{pmatrix}
\end{equation} where $\tilde\bX=[\bX,\bx]$. For any $\lambda\geq 0$,
 both $\tilde\bPsi$ and $\bPsi$ are positive definite because of \eqref{eq:Phil_approx} and \eqref{eq:Psi_ell}. Following the proof of Lemma~\ref{lemm:Kmix_bound}, we can similarly derive that the Schur complement $\E_{\bbeta}[\tau(\bbeta^\top\bx)^2]+\lambda-\bK_{\tau,m}^\top\bPsi^{-1}_{\lambda}\bK_{\tau,m}$ is positive, where $\E_{\bbeta}[\tau(\bbeta^\top\bx)^2]=\norm{\tau}_2^2$. Therefore, we have
\begin{align}
    &\norm{\E_{\bbeta}[\tau(\bbeta^\top\bx)\tau(\bbeta^\top\bX)]\bK_\lambda^{-\frac{1}{2}}}^2=\bK_{\tau,m}^\top\bK_\lambda^{-1}\bK_{\tau,m}\\
    =~& \bK_{\tau,m}^\top\bPsi_{\lambda}^{-\frac{1}{2}}\bPsi_{\lambda}^{\frac{1}{2}}\bK_\lambda^{-1}\bPsi^{\frac{1}{2}}_{\lambda}\bPsi^{-\frac{1}{2}}_{\lambda}\bK_{\tau,m}\\
    \le~ & \bK_{\tau,m}^\top\bPsi^{-1}_{\lambda}\bK_{\tau,m}\cdot \norm{\bPsi^{\frac{1}{2}}_{\lambda}\bK_\lambda^{-1}\bPsi^{\frac{1}{2}}_{\lambda}}\le (\lambda+\norm{\tau}_2^2)\norm{\bPsi^{\frac{1}{2}}_{\lambda}\bK_\lambda^{-1}\bPsi^{\frac{1}{2}}_{\lambda}}.
\end{align}
Additionally, following the same proof of Lemma~\ref{lemm:KPsiK}, we can also obtain $\norm{\bPsi_{\lambda}^{\frac{1}{2}}\bK_\lambda^{-1}\bPsi_{\lambda}^{\frac{1}{2}}}\le C$, for some constant $C>0$ which only depends on $\sigma,\tau$. This concludes the proof.
\end{proof}

The following lemma is analogous to Lemma 5 in \cite{montanari2020interpolation}, which addresses the concentrations of $\bK_N^{(2)}$ and $f^*(\bx)\bK_N(\bX,\bx)$ respectively.
\begin{lemma}\label{lemm:deltai}
Suppose that the assumptions of Theorem~\ref{thm:test_diff} hold. 
For any $\lambda\ge 0$, define
\begin{align*}
     \delta_1:=&  \E_{\bbeta, \varepsilon}\left[\by^\top\Kl^{-1} \left(\KxxN-\Kxx\right)\Kl^{-1}\by\right],\\
    \delta_{2}:= & \E_{\bbeta,\bx}\left[f^*(\bx)\left(\bK_N(\bx,\bX)-\bK(\bx,\bX)\right)\Kl^{-1}f^*(\bX)\right].
\end{align*}
Then, for sufficiently large $n$,  with probability at least $1-\frac{1}{2}\log^{C_\sigma} (N)$, when $n\geq \max \{n_0,n_1\}$,  there exists some constant $C>0$ depending only on $\sigma,\tau$ such that
\begin{align}
    |\delta_1|\le & C(1+\lambda)\log^{C_\sigma}(N) \sqrt{\frac{n}{N}},    \label{eq:bound_delta1} \\ 
    |\delta_2|\le & C(1+\lambda)\log^{C_\sigma}(N) \sqrt{\frac{n}{N}}.   \label{eq:bound_delta2}
\end{align}

\end{lemma}
\begin{proof} Let $\bv:=\Kl^{-1}\by$ and $\tilde{g}(\bx):=\Km^\top\bv$.
Notice that $\delta_1=\delta_{1,1}+\delta_{1,2}$, where
\begin{align*}
    \delta_{1,1}:=~&\E_{\bbeta, \varepsilon}[\bv^\top \E_\bx\left[(\KmN-\Km)(\KmN-\Km)^\top\right]\bv],\\
    \delta_{1,2}:=~&2\E_{\bbeta, \varepsilon}[\bv^\top \E_\bx\left[(\KmN-\Km)\Km^\top\right]\bv]=2\E_{\bbeta, \varepsilon,\bx}[\bv^\top(\KmN-\Km)\tilde{g}(\bx)],
\end{align*} Taking expectation with respect to $\bW$, we can obtain
\begin{align}
    0\le\E[\delta_{1,1}]=~&\frac{1}{N^2}\sum_{i,j=1}^N \E_{\bbeta, \varepsilon}\left[\bv^\top \E_{\bW,\bx}\left[\left(\sigma(\bw_i^\top\bx)\sigma(\bw_i^\top\bX)^\top-\Km\right)\left(\sigma(\bw_j^\top\bx)\sigma(\bw_j^\top\bX)-\Km^\top\right)\right]\bv\right]\\
    =~&\frac{1}{N^2}\sum_{i=1}^N  \E_{\bbeta, \varepsilon}\left[\bv^\top \E_{\bW,\bx}\left[\left(\sigma(\bw_i^\top\bx)\sigma(\bw_i^\top\bX)^\top-\Km\right)\left(\sigma(\bw_i^\top\bx)\sigma(\bw_i^\top\bX)-\Km^\top\right)\right]\bv\right]\\
     =~&\frac{1}{N^2}\sum_{i=1}^N  \E_{\bbeta, \varepsilon}\left[\bv^\top \E_{\bW,\bx}\left[\sigma(\bw_i^\top\bx)^2\sigma(\bw_i^\top\bX)^\top\sigma(\bw_i^\top\bX)-\Km\Km^\top\right]\bv\right]\\
     \le~ &\frac{1}{N}  \E_{\bbeta, \varepsilon}\left[\bv^\top \E_{\bw,\bx}\left[\sigma(\bw^\top\bx)^2\sigma(\bw^\top\bX)^\top\sigma(\bw^\top\bX)\right]\bv\right],\label{eq:delta1_first}
\end{align}
where in the last line we apply the fact $\Km=\E_{\bw_i}[\sigma(\bw_i^\top\bx)\sigma(\bw_i^\top\bX)^\top]\in\R^n$ for any $i$-th row of $\bW$. 

Furthermore, by applying Lemma~\ref{lemm:KPsiK} and Cauchy–Schwartz inequality, we have
\begin{align}
    &\E_{\bbeta, \varepsilon}\left[\bv^\top \E_{\bw,\bx}\left[\sigma(\bw^\top\bx)^2\sigma(\bw^\top\bX)^\top\sigma(\bw^\top\bX)\right]\bv\right]\\
    =~&\Tr\left(\Kl^{-1}\E_{\bw,\bx}\left[\sigma(\bw^\top\bx)^2\sigma(\bw^\top\bX)^\top\sigma(\bw^\top\bX)\right]\Kl^{-1} \left(\bPsi+\sigma^2_\eps\Id\right)\right)\\
    =~& \E_{\bw,\bx}\left[\sigma(\bw^\top\bx)^2\sigma(\bw^\top\bX)\Kl^{-1}\left(\bPsi+\sigma^2_\eps\Id\right)\Kl^{-1}\sigma(\bw^\top\bX)^\top\right]\\
    \le~ & \norm{\sigma}_4^2\E_{\bw,\bx}\left[\norm{\sigma(\bw^\top\bX)}^4\norm{\Kl^{-1}\left(\bPsi+\sigma^2_\eps\Id\right)\Kl^{-1}}^2\right]^{\frac{1}{2}}\\
    \le~ &  \norm{\sigma}_4^2\cdot\frac{1}{\lambda_0}\left(C+\frac{\sigma^2_\eps}{\lambda_0}\right)\E_{\bw}\left[\norm{\sigma(\bw^\top\bX)}^4\right]^{\frac{1}{2}}\\
    = ~& \norm{\sigma}_4^2\cdot\frac{1}{\lambda_0}\left(C+\frac{\sigma^2_\eps}{\lambda_0}\right)\E_{\bw}\left[\left(\sum_{i=1}^n\sigma(\bw^\top\bx_i)^2\right)^2\right]^{\frac{1}{2}}\\
    \le &\norm{\sigma}_4^4\cdot\frac{1}{\lambda_0}\left(C+\frac{\sigma^2_\eps}{\lambda_0}\right)\cdot n\label{eq:delta1_second}
\end{align}where $\bw\sim \cN(0,\Id)$ is independent of $\bx$ and $ \norm{\sigma}_4^4=\E_{\bw,\bx}[\sigma(\bw^\top\bx)^4]$. Therefore, combining \eqref{eq:delta1_first} and \eqref{eq:delta1_second}, we can conclude that
\begin{equation}
    \E[|\delta_{1,1}|]\le C_{1,1}\frac{n}{N},
\end{equation}for some constant $C_{1,1}>0$ which only relies on $\sigma$ and $\sigma_{\eps}$. Then, Markov inequality deduces that for sufficiently large $n$,
\begin{equation}\label{eq:delta_11_prob}
    \Parg{|\delta_{1,1}|>8C_{1,1}\log ^{C_\sigma}(N)\frac{n}{N}}\le \frac{1}{8\log^{C_\sigma} (N)}.
\end{equation} 

Next, we consider  $\delta_2$. Let $\bz_1,\bz_2$ be two i.i.d. copies of $\bx$, and $\bbeta_1,\bbeta_2$ be two  i.i.d. copies of $\bbeta$. Let $\bu_i:=\bK^{-1}_\lambda \tau(\bbeta_i^\top\bX)^\top$ and $g_i(\bx):=\tau(\bbeta^\top_i\bx)$ for $i=1,2$. Notice that $\E_{\bw_i}[\sigma(\bw_i^\top\bz_1)\sigma(\bX^\top\bw_i)]=\bK(\bX,\bz_1)$. Then, taking expectation with respect to $\bW$, we can obtain
\begin{align}
    &\E[\delta^2_2]=  \E_{\bbeta_1,\bbeta_2}\left[\bu_1^\top \E_{\bW,\bz_1,\bz_2}[\left(\bK_N(\bX,\bz_1)-\bK(\bX,\bz_1)\right)g_1(\bz_1)g_2(\bz_2)\left(\bK_N(\bz_2,\bX)-\bK(\bz_2,\bX)\right)]\bu_2\right]\\
    =~ & \frac{1}{N^2}\sum_{i,j=1}^N \E\left[\bu_1^\top \left(\sigma(\bw_i^\top\bz_1)\sigma(\bX^\top\bw_i)-\bK(\bX,\bz_1)\right)g_1(\bz_1)g_2(\bz_2)\left(\sigma(\bw_j^\top\bz_2)\sigma(\bw_j^\top\bX)-\bK(\bz_2,\bX)\right)\bu_2\right]\\
    = ~& \frac{1}{N^2}\sum_{i=1}^N \E \left[g_1(\bz_1)g_2(\bz_2)\sigma(\bw_i^\top\bz_1)\sigma(\bw_i^\top\bz_2)\bu_1^\top\left(\sigma(\bX^\top\bw_i)\sigma(\bw_i^\top\bX)-\bK(\bX,\bz_1)\bK(\bz_2,\bX)\right)\bu_2\right]\\
    \le ~ & \frac{1}{N}\E_{\bbeta_1,\bbeta_2,\bw,\bz_1,\bz_2}\left[g_1(\bz_1)g_2(\bz_2)\sigma(\bw^\top\bz_1)\sigma(\bw^\top\bz_2)\cdot\bu_1^\top\sigma(\bX^\top\bw)\sigma(\bw^\top\bX)\bu_2\right],
\end{align}
where in the last line, we apply the following bound:
\begin{align}
     &\E_{\bbeta_1,\bbeta_2,\bz_1,\bz_2}\left[ g_1(\bz_1)g_2(\bz_2)\sigma(\bw_i^\top\bz_1)\sigma(\bw_i^\top\bz_2)\bu_1^\top\bK(\bX,\bz_1)\bK(\bz_2,\bX)\bu_2\right]\\
     =&\left(\E_{\bbeta,\bx}\left[\tau(\bbeta^\top\bx)\sigma(\bw_i^\top\bx)\tau(\bbeta^\top\bX)\bK_\lambda^{-1}\bK(\bX,\bx)\right]\right)^2\ge 0.
\end{align}
Let $\bv_i:=\bK_\lambda^{-1/2}\E_{\bbeta_i}[\tau(\bbeta_i^\top\bz_i)\tau(\bbeta_i^\top\bX)]^\top$ for $i=1,2$. Then, Lemma~\ref{lemm:KPsiK_vector} shows that $\|\bv_i\|\le C(1+\lambda)$ for some universal constant $C$. Thus, similarly with the derivation of \eqref{eq:delta1_second}, we can deduce that
\begin{align}
&\E_{\bbeta_1,\bbeta_2,\bw,\bz_1,\bz_2}\left[g_1(\bz_1)g_2(\bz_2)\sigma(\bw^\top\bz_1)\sigma(\bw^\top\bz_2)\cdot\bu_1^\top\sigma(\bX^\top\bw)\sigma(\bw^\top\bX)\bu_2\right]\\
  = ~& \E_{\bw,\bz_1,\bz_2}\left[\sigma(\bw^\top\bz_1)\sigma(\bw^\top\bz_2)\cdot \bv_1\bK_\lambda^{-1/2}\sigma(\bX^\top\bw)\sigma(\bw^\top\bX)\bK_\lambda^{-1/2}\bv_2\right]\\
  \le~ & \E_{\bw,\bz_1,\bz_2}\left[\sigma(\bw^\top\bz_1)^2\sigma(\bw^\top\bz_2)^2\right]^{\frac{1}{2}}\E_{\bw,\bz_1,\bz_2}\left[\norm{\bv_1}^2\norm{\bv_2}^2\norm{\bK_\lambda^{-1/2}\sigma(\bX^\top\bw)\sigma(\bw^\top\bX)\bK_\lambda^{-1/2}}^2\right]^{\frac{1}{2}}\\
  \le ~& C^2(1+\lambda)^2 \norm{\sigma}_4^2 \E_\bw\left[\left( \sigma(\bw^\top\bX)\bK_\lambda^{-1}\sigma(\bX^\top\bw)\right)^2\right]^{\frac{1}{2}}\\
  \le ~&\frac{C^2(1+\lambda)^2 \norm{\sigma}_4^2}{\lambda_0} \E_\bw\left[\left(\sum_{i=1}^n\sigma(\bw^\top\bx_i)^2\right)^2\right]^{\frac{1}{2}}\\
  \le ~&\frac{C^2(1+\lambda)^2 \norm{\sigma}_4^4n}{\lambda_0},
\end{align}where the last line is analogous to \eqref{eq:delta1_second}. Therefore, $\E[\delta_{2}^2]\le C_2(1+\lambda)^2\frac{n}{N}$. This indicates, for any $t>0$,
\begin{equation}
    \Parg{|\delta_2|>2t(1+\lambda)}\le \frac{C_2n}{Nt^2}.
\end{equation}Hence, by taking $t=4\log^{C_\sigma}(N)\sqrt{C_2n/N}$, we can conclude the bound of $\delta_2$ in \eqref{eq:bound_delta2} with probability at least $1-\frac{1}{8}\log^{-2{C_\sigma}}(N)$. 

The analysis of $\delta_{1,2}$ is similar to the analysis for $\delta_2$.  By definition, we have
\begin{align}
    \delta_{1,2}=~&2\Tr \left(\E_\bx\left[(\KmN-\Km)\Km^\top\right]\bK_{\lambda}^{-1}(\bPsi+\sigma^2_\eps\Id)\Kl^{-1}\right)\\
    =~&2\E_\bx\left[\Km^\top\bK_{\lambda}^{-1}(\bPsi+\sigma^2_\eps\Id)\Kl^{-1}(\KmN-\Km)\right].
\end{align}
Then, consider $\bz_1,\bz_2$ as i.i.d. copies of $\bx$. Let $\bA:=\bK_{\lambda}^{-1}(\bPsi+\sigma^2_\eps\Id)\Kl^{-1}$ and $$\bK_{m,i}:=\E_\bw[\sigma(\bw^\top\bz_i)\sigma(\bw^\top\bX)]^\top\in\R^n,$$
for $i=1,2$. Then, we have
\begin{align}
    &\E[\delta_{1,2}^2]=\frac{4}{N^2}\sum_{i,j=1}^N\E\left[\bK_{m,1}^\top\bA\left(\sigma(\bw_i^\top\bz_1)\sigma(\bX^\top\bw_i)-\bK_{m,1}\right)\left(\sigma(\bw_j^\top\bz_2)\sigma(\bw_j^\top\bX)-\bK_{m,2}^\top\right)\bA\bK_{m,2} \right]
    \\
   =~& \frac{4}{N^2}\sum_{i=1}^N\E\left[\bK_{m,1}^\top\bA\left(\sigma(\bw_i^\top\bz_1)\sigma(\bX^\top\bw_i)-\bK_{m,1}\right)\left(\sigma(\bw_i^\top\bz_2)\sigma(\bw_i^\top\bX)-\bK_{m,2}^\top\right)\bA\bK_{m,2} \right]
    \\
    =~& \frac{4}{N^2}\sum_{i=1}^N\E\left[\bK_{m,1}^\top\bA\left(\sigma(\bw_i^\top\bz_1)\sigma(\bw_i^\top\bz_2)\sigma(\bX^\top\bw_i)\sigma(\bw_i^\top\bX)-\bK_{m,1}\bK_{m,2}^\top\right)\bA\bK_{m,2} \right]
    \\
    \overset{(i)}{\le}~&  \frac{4}{N}\E_{\bw,\bz_1,\bz_2}\left[\sigma(\bw^\top\bz_1)\sigma(\bw^\top\bz_2)\bK_{m,1}^\top\bA\sigma(\bX^\top\bw)\sigma(\bw^\top\bX)\bA\bK_{m,2} \right]\\
    \le ~& \frac{4}{N}\E\left[\sigma(\bw^\top\bz_1)^2\sigma(\bw^\top\bz_2)^2\right]^{\frac{1}{2}}\E\left[\norm{\bK_{m,1}^\top\bK_\lambda^{-\frac{1}{2}}}^2\norm{\bK_{m,2}^\top\bK_\lambda^{-\frac{1}{2}}}^2\norm{\bK_\lambda^{\frac{1}{2}}\bA\sigma(\bX^\top\bw)\sigma(\bw^\top\bX)\bA\bK_\lambda^{\frac{1}{2}}}^2 \right]\\
    \overset{(ii)}{\le} ~& \frac{4C^2(1+\lambda)^2 \norm{\sigma}_4^2}{\lambda_0N}\E\left[\left(\sum_{i=1}^n\sigma(\bw^\top\bx_i)^2\right)^2\right]^{1/2}\le C_{1,2}(1+\lambda)^2\frac{n}{N},
\end{align}
for some constant $C_{1,2}>0$, where $(i)$ is because of positiveness of $\bA$ and $(ii)$ is due to Lemmas~\ref{lemm:KPsiK} and \ref{lemm:KPsiK_vector}. Thus, Markov inequality allows us to obtain for a constant $C>0$,
\begin{equation} 
\Parg{|\delta_{1,2}|>C(1+\lambda)\log^{{C_\sigma}}(N)\sqrt{\frac{n}{N}}}\le \frac{1}{8\log^{2{C_\sigma}} (N)},
\end{equation}
Together with \eqref{eq:delta_11_prob}, we can conclude the bound for $\delta_1$ in \eqref{eq:bound_delta1}.
\end{proof}

Based on the above lemmas, we are now ready to prove Theorem~\ref{thm:test_diff} for the concentrations of the generalization errors between RFRR and KRR.
\begin{proof}[Proof of Theorem~\ref{thm:test_diff}]
Recall $\bK=\bK(\bX,\bX)$ and $\bK_N=\bK_N(\bX,\bX)$. Hence, we can further decompose the test errors \eqref{eq:def_testerror} for both RFRR and KRR in the following way:
\begin{align}
    \mathcal L(\hat f^{(\KK)}_\lambda)=~& \E [|f^*(\bx)|^2]+\Tr\left[(\bK+\lambda\Id)^{-1}\E[\by\by^\top] (\bK+\lambda\Id)^{-1}\E[\bK(\bX,\bx) \bK(\bx,\bX)]\right]\nonumber\\
    &-2\Tr\left[(\bK+\lambda\Id)^{-1}\E[\by f^*(\bx)\bK(\bx,\bX)]\right],\label{eq:test_K}\\
    \mathcal L(\hat f^{(\RF)}_\lambda)=~& \E [|f^*(\bx)|^2]+\Tr\left[(\bK_N+\lambda\Id)^{-1}\E[\by\by^\top] (\bK_N+\lambda\Id)^{-1}\E[\bK_N(\bX,\bx) \bK_N(\bx,\bX)]\right]\nonumber\\
    &-2\Tr\left[(\bK_N+\lambda\Id)^{-1}\E[\by f^*(\bx)\bK_N(\bx,\bX)]\right],\label{eq:test_RF}
\end{align}
where we are taking expectations with respect to $\bx,\bbeta$, and $\beps$. Let us denote
\begin{align*}
        E_1:=~&\Tr\left[(\bK_N+\lambda\Id)^{-1}\E[\by\by^\top] (\bK_N+\lambda\Id)^{-1}\E[\bK_N(\bX,\bx) \bK_N(\bx,\bX)]\right],\\
        \bar E_1:=~&\Tr\left[(\bK+\lambda\Id)^{-1}\E[\by\by^\top] (\bK+\lambda\Id)^{-1}\E[\bK(\bX,\bx) \bK(\bx,\bX)]\right],\\
       E_2:=~&\Tr\left[(\bK_N+\lambda\Id)^{-1}\E[\by f^*(\bx)\bK_N(\bx,\bX)]\right],\\
        \bar E_2:=~&\Tr\left[(\bK+\lambda\Id)^{-1}\E[\by f^*(\bx)\bK(\bx,\bX)]\right].
\end{align*}   
Therefore, by taking the expectation with respect to $\bbeta$ and $\varepsilon$, we have
\begin{align*}
     E_1=~& \E_\bx\left[\bK_N(\bx,\bX)(\bK_N+\lambda\Id)^{-1}\left(\bPsi+\sigma^2_\eps\Id\right)(\bK_N+\lambda\Id)^{-1}\bK_N(\bX,\bx)  \right],\\
     \bar E_1=~&\E_\bx\left[\bK(\bx,\bX)(\bK+\lambda\Id)^{-1}\left(\bPsi+\sigma^2_\eps\Id\right)(\bK+\lambda\Id)^{-1}\bK(\bX,\bx)  \right],\\
E_2=~&\Tr\left[(\bK_N+\lambda\Id)^{-1}\E[\bu\bK_N(\bx,\bX)]\right],\\
    \bar E_2=~&\Tr\left[(\bK+\lambda\Id)^{-1}\E[\bu\bK(\bx,\bX)]\right],
\end{align*}
where $\bPsi=\E_\bbeta[\tau(\bX^\top\bbeta)\tau(\bbeta^{\top}\bX)]$ and  $\bu=\E_\bbeta[\tau(\bX^\top\bbeta) f^*(\bx)]\in\R^n$.  We can further get the decomposition: $E_1-\bar E_1=J_{1,1}+J_{1,2}+J_{1,3}$ and $E_2-\bar E_2= J_{2,1}+J_{2,2}$, where
 \begin{align*}
    J_{1,1}:=~& \E_\bx\left[\KmN^\top\left(\KNl^{-1}-\Kl^{-1}\right)\left(\bPsi+\sigma^2_\eps\Id\right)\KNl^{-1}\KmN\right],\\
     J_{1,2}:=~& \E_\bx\left[\KmN^\top\Kl^{-1}\left(\bPsi+\sigma^2_\eps\Id\right)\left(\KNl^{-1}-\Kl^{-1}\right)\KmN\right],\\
   J_{1,3}:=~& \E_{\bbeta,\beps}\left[ \by^\top\Kl^{-1} \left(\KxxN-\Kxx\right)\Kl^{-1}\by\right],\\
    J_{2,1}:=~&  \E_\bx\left[\bK_N(\bx,\bX)\left(\KNl^{-1}-\Kl^{-1}\right)\bu\right],\\
    J_{2,2}:= ~&  \E_\bx\left[\left(\bK_N(\bx,\bX)-\bK(\bx,\bX)\right)\Kl^{-1}\bu\right].
\end{align*}
Recall that $\bPsi=\E_\bbeta[\tau(\bX^\top\bbeta)\tau(\bbeta^{\top}\bX)]$, $\bu=\E_\bbeta[\tau(\bX^\top\bbeta) f^*(\bx)]\in\R^n$ and  $\bPsi_{\lambda} =\bPsi+\lambda\Id$. Notice that 
\begin{align}
    \KNl^{-1}-\Kl^{-1}=~&\KNl^{-1}\left(\bK-\bK_N\right)\Kl^{-1}\\
    =~& \KNl^{-\frac{1}{2}} \KNl^{-\frac{1}{2}}\Kl^{\frac{1}{2}}\Kl^{-\frac{1}{2}}\left(\bK-\bK_N\right)\Kl^{-\frac{1}{2}}\Kl^{-\frac{1}{2}}.
\end{align}
Hence, we can apply Proposition~\ref{prop:kernel_con}, Corollary~\ref{coro_con}, Lemmas~\ref{lemm:Kmix_bound}, \ref{lemm:KPsiK} and \ref{lemm:KPsiK_vector} to conclude that 
\begin{align*}
    |J_{1,1}|\le ~& \E_\bx\left[\norm{\KmN^\top\KNl^{-1/2}}^2\right]\cdot\norm{\KNl^{-1/2}\Kl^{1/2}}^2\norm{\Kl^{-1/2}\left(\bK-\bK_N\right)\Kl^{-1/2}}\cdot\norm{\Kl^{-1/2}\left(\bPsi+\sigma^2_\eps\Id\right)\Kl^{-1/2}}\\
    \le ~& C(1+\lambda)\log^{C_{\sigma}}(N)\sqrt{\frac{n}{N}},\\
    |J_{1,2}|\le~ & \E_\bx\left[\norm{\KmN^\top\KNl^{-1/2}}^2\right]\cdot\norm{\KNl^{-1/2}\Kl^{1/2}}\norm{\KNl^{1/2}\Kl^{-1/2}}\\
     &\cdot\norm{\Kl^{-1/2}\left(\bK-\bK_N\right)\Kl^{-1/2}}\norm{\Kl^{-1/2}\left(\bPsi+\sigma^2_\eps\Id\right)\Kl^{-1/2}}\\
    \le ~& C(1+\lambda)\log^{C_{\sigma}}(N)\sqrt{\frac{n}{N}},\\
    |J_{2,1}|\le ~& \E_\bx\left[\left|\KmN^\top\KNl^{-1/2}\KNl^{-1/2}\Kl^{1/2}\Kl^{-1/2}(\bK-\bK_N)\Kl^{-1/2} \Kl^{-1/2}\bu\right|\right]\\
    \le~ & \E_\bx\left[\left\|\KmN^\top\KNl^{-1/2}\right\|\cdot\norm{\KNl^{-1/2}\Kl^{1/2}}\cdot\norm{\Kl^{-1/2}(\bK-\bK_N)\Kl^{-1/2}}\norm{ \Kl^{-1/2}\bu}\right]\\
    \le ~& C(1+\lambda)\log^{C_{\sigma}}(N)\sqrt{\frac{n}{N}},
\end{align*}for some constant $C>0$ depending on the norms of $\tau$ and $\sigma$, $\lambda_0$ and $\sigma_{\eps}$
 with probability at least $1-N^{-1}$. 

 Meanwhile, based on Lemma~\ref{lemm:deltai}, $|J_{1,3}| $ and $|J_{2,2}| $ are both less than $C\log^{C_\sigma}(N)\sqrt{\frac{n}{N}}$ with probability at least $1-\frac{1}{2}\log^{-C_\sigma}(N)$, because $\delta_1=J_{1,3}$ and $\delta_2=J_{2,2}$. Hence, combining the controls of $J_{1,1}, J_{1,2}, J_{1,3}$ and $J_{2,1}, J_{2,2}$, we complete the proof of Theorem~\ref{thm:test_diff}.
\end{proof}

\subsection{Proof of Theorem~\ref{thm:pkrr}}
We first show \eqref{eq:Kl_train}. In the proof of Proposition~\ref{prop:pkrr}, we know $\lambda_{\min}(\bK_\ell)\ge 2\lambda_0$ and $\lambda_{\min}(\bK)\ge \lambda_0$. Similar to the proof of \eqref{eq:last_thm37}, using the closed form formula of the training error from \eqref{eq:Etrain_Kn}
 and Proposition \ref{prop:pkrr}, we have
\begin{align}
     \left| E_{\train}^{(\ell,\lambda)}-E_{\train}^{(\KK,\lambda)}\right| 
     =& \frac{\lambda^2}{n} \left|\v y^\top\left[(\bK_{\ell}+\lambda\Id)^{-2}-(\bK+\lambda\Id)^{-2}\right] \v y\right| \notag \\
     \leq & \frac{\lambda^2}{n} \|(\bK_{\ell}+\lambda\Id)^{-2}-(\bK+\lambda\Id)^{-2}\|\cdot  \|\v y\|^2 \notag\\
     \leq &\frac{3\lambda^2 \|\by\|^2}{2\lambda_0 n} \|(\bK_{\ell}+\lambda\Id)^{-1}-(\bK+\lambda\Id)^{-1}\| \notag \\
     \leq &\frac{3\lambda^2\|\by\|^2}{2 \lambda_0^3 n} \|\bK-\bK_{\ell}\|\leq  \frac{C\lambda^2\|\by\|^2\|\sigma\|_4^2}{\lambda_0^3 n}   \left\| \left(\bX^\top \bX\right)^{\odot \ell+1}-\Id\right\|_F
     \label{eq:719}
\end{align}
for an absolute constant $C>0$,
where in the third inequality, we use the estimate
\begin{align}\label{eq:Kl_lambda_approx}
 \|(\bK_{\ell}+\lambda\Id)^{-1}-(\bK+\lambda\Id)^{-1}\|=\norm{\bK_{\lambda}^{-1} (\bK-\bK_{\ell})\bK_{\ell,\lambda}^{-1}} \leq \frac{1}{(\lambda+\lambda_0)\lambda_0} \norm{\bK-\bK_{\ell}} . 
\end{align}

Next, we prove \eqref{eq:GCVn}. 
With the same proof in Lemma \ref{lemm:K_lambda_tr}, we also have
\begin{align}\label{eq:TRKL}
    \left(\lambda+\norm{\sigma}_2^{2}\right)^{-1}\le \tr\bK_{\ell,\lambda}^{-1}\le \lambda_0^{-1}.
\end{align}
From the definition of GCV in \eqref{eq:GCV}, we have
\begin{align}
     |\GCV_n^{(\KK,\lambda)}-\GCV_n^{(\ell,\lambda)}|\leq &\lambda^{-2}\left|\left(\left(\tr \bK_{\lambda}^{-1}\right)^{-2}-\left(\tr \bK_{\ell,\lambda}^{-1}\right)^{-2}\right) E_{\train}^{(\KK,\lambda)}\right| \label{eq:GCVl_1}\\
      &+\lambda^{-2}\left|\left(\tr \bK_{\lambda}^{-1}\right)^{-2}\left( E_{\train}^{(\KK,\lambda)}-E_{\train}^{(\ell,\lambda)}\right)\right|.\label{eq:GCVl_2}
\end{align}
Equipped with   \eqref{eq:TRKL} and Lemma \ref{lemm:K_lambda_tr}, following every step in the proof of \eqref{eq:GCV_K} in Section~\ref{sec:LOOCV}, we can obtain 
 a similar bound for \eqref{eq:GCVl_1} as follows:
\begin{align}
  &\lambda^{-2}\left|\left(\left(\tr \bK_{\lambda}^{-1}\right)^{-2}-\left(\tr \bK_{\ell,\lambda}^{-1}\right)^{-2}\right) E_{\train}^{(\KK,\lambda)}\right|\\
  \leq & \left(\tr\bK_{\lambda}^{-1}\right)^{-2}\left(\tr\bK_{\ell,\lambda}^{-1}\right)^{-2} \left|\tr (\bK_{\lambda}^{-1}-\bK_{\ell,\lambda}^{-1})\right|\tr \left(\bK_{\lambda}^{-1}+\bK_{\ell,\lambda}^{-1}\right)\frac{1}{n}\|\bK_{\lambda}^{-2}\| \norm{\by}^2  \\
  \leq & \frac{8(\lambda+\norm{\sigma}_2^2)^4}{\lambda_0^5 n}\|\bK-\bK_{\ell}\|\leq \frac{8\sqrt{2}(\lambda+\norm{\sigma}_2^2)^4\|\sigma\|_4^2}{\lambda_0^5 n} \left\| \left(\bX^\top \bX\right)^{\odot \ell+1}-\Id\right\|_F.
\end{align}
Similarly, for the second term \eqref{eq:GCV_2}, we have from \eqref{eq:719} and Lemma \ref{lemm:K_lambda_tr},
\begin{align*}
    \lambda^{-2}\left|\left(\tr \bK_{\lambda}^{-1}\right)^{-2}\left( E_{\train}^{(\KK,\lambda)}-E_{\train}^{(\ell,\lambda)}\right)\right| &\leq \frac{C(\lambda+\norm{\sigma}_2^2)^{2} \|\sigma\|_4^2}{\lambda_0^3 n }\left\| \left(\bX^\top \bX\right)^{\odot \ell+1}-\Id\right\|_F, 
\end{align*}
which implies \eqref{eq:GCVn}. Next, we verify \eqref{eq:CVn}. Recall \eqref{eq:shortcut} and \eqref{eq:shortcut_N}. Analogously, we have 
\begin{align}
    \CV_n^{(\ell,\lambda)}= & \frac{1}{n}\by^\top \bK_{\ell,\lambda}^{-1}\bD_\ell^{-2} \bK_{\ell,\lambda}^{-1}\by,\label{eq:shortcut_ell}
\end{align}
where $\bD_\ell$ is a diagonal matrix with diagonals $[\bD_\ell]_{ii}=[\bK_{\ell,\lambda}^{-1}]_{ii}$ for $i\in[n]$. Notice that $\norm{\bD_\ell-\bD}$ has the same upper bound as \eqref{eq:Kl_lambda_approx}, and any $[\bD_\ell]_{ii}$ has the same lower and upper bounds as \eqref{eq:TRKL} for $i\in[n]$. Hence, repeatedly applying Proposition \ref{prop:pkrr} and following \eqref{eq:CV_derive}, we can obtain
\begin{align*}
    \left|\CV_n^{(\ell,\lambda)}-\CV_n^{(\KK,\lambda)}\right|\le C_2(1+\lambda^4)\frac{\|\by\|^2}{n}\left\| \left(\bX^\top \bX\right)^{\odot \ell+1}-\Id\right\|_F,
\end{align*}for some constant $C_2>0$ which only relies on $\norm{\sigma}_2, \norm{\sigma}_4$, and $\lambda_0$. This  concludes the bound in \eqref{eq:CVn}.

Finally, we can repeat the analysis in the proof of Theorem~\ref{thm:test_diff} and apply \eqref{eq:Kl_lambda_approx} to obtain \eqref{eq:Kl_test}.
By taking expectation with respect to $\bbeta$ and $\varepsilon$, we have $\left| \mathcal L(\hat{f}_{\lambda}^{(\ell)}(\x))-\mathcal L (\hat{f}_{\lambda}^{(\KK)}(\x))\right| \le |E_1'-\bar E_1|+ |E_2'-\bar E_2|$, where
\begin{align*}
     E_1':=& \E_\bx\left[\bK_\ell(\bx,\bX)\bK_{\ell,\lambda}^{-1}\left(\bPsi+\sigma^2_\eps\Id\right)\bK_{\ell,\lambda}^{-1}\bK_\ell(\bX,\bx)  \right],\\
     \bar E_1=&\E_\bx\left[\bK(\bx,\bX)\Kl^{-1}\left(\bPsi+\sigma^2_\eps\Id\right)\Kl^{-1}\bK(\bX,\bx)  \right],\\
     E_2':=&\Tr\left[\bK_{\ell,\lambda}^{-1}\E[\bu\bK_\ell(\bx,\bX)]\right],\\
    \bar E_2=&\Tr\left[\bK_{\lambda}^{-1}\E[\bu\bK(\bx,\bX)]\right].
\end{align*} Denote 
$\bK_{m,\ell}=\bK_\ell(\bX,\bx)$ and $\bK_{\ell,\lambda}=\lambda\Id+\bK_\ell(\bX,\bX).$
Recall that $\bPsi=\E_\bbeta[\tau(\bX^\top\bbeta)\tau(\bbeta^{\top}\bX)]$ and  $\bu=\E_\bbeta[\tau(\bX^\top\bbeta) f^*(\bx)]\in\R^n$. Because of the Assumption \ref{assump:testdata}, similar to the proof of Proposition \ref{prop:pkrr}, we obtain 
\begin{equation}\label{eq:E_1_1}
    \norm{\bK_{m,\ell}-\bK_{m}}\le  \norm{\bK_{m,\ell}-\bK_{m}}_2\leq \sqrt{2} \|\sigma\|_4^2 \left\| (\bX^\top \bx)^{\odot (\ell+1)}\right\|_2\leq \frac{1}{\sqrt 2} \lambda_0.
\end{equation}
Moreover, analogously to Lemma~\ref{lemm:Kmix_bound}, we have
\begin{equation}\label{eq:E_1_2}
    \norm{\bK_{m,\ell}^\top\bK_{\ell,\lambda}^{-1/2}}\le \norm{\sigma}_2^2+\lambda.
\end{equation}
Also, following the proofs of Lemma~\ref{lemm:KPsiK} and Lemma~\ref{lemm:KPsiK_vector}, we can check that 
\begin{equation}\label{eq:E_1_3}
    \norm{\bK_{\ell,\lambda}^{-1/2}\bPsi\bK_{\ell,\lambda}^{-1/2}},~\norm{\bK_{\ell,\lambda}^{-1/2}\bu}\le C,
\end{equation}for some  constant $C>0$ depending only on $\sigma,\tau$. Therefore, because of Proposition~\ref{prop:pkrr}, Lemmas~\ref{lemm:Kmix_bound} and \ref{lemm:KPsiK}, and \eqref{eq:E_1_1}, \eqref{eq:E_1_2} and \eqref{eq:E_1_3},  we can deduce that  
\begin{align}
    |E_1'-\bar E_1|\le~ & \left| (\bK_{m,\ell}-\Km)^\top \bK_{\ell,\lambda}^{-1/2}\bK_{\ell,\lambda}^{-1/2}\left(\bPsi+\sigma^2_\eps\Id\right)\bK_{\ell,\lambda}^{-1/2}\bK_{\ell,\lambda}^{-1/2}\bK_{m,\ell}\right|\\
    &+\left| \Km^\top \bK_{\lambda}^{-1/2}\bK_{\lambda}^{-1/2}\left(\bK-\bK_\ell\right)\bK_{\ell,\lambda}^{-1/2}\bK_{\ell,\lambda}^{-1/2}\left(\bPsi+\sigma^2_\eps\Id\right)\bK_{\ell,\lambda}^{-1/2}\bK_{\ell,\lambda}^{-1/2}\bK_{m,\ell}\right|\\
    &+\left| \Km^\top \bK_{\lambda}^{-1/2}\bK_{\lambda}^{-1/2}\left(\bPsi+\sigma^2_\eps\Id\right)\bK_{\lambda}^{-1/2}\bK_{\lambda}^{-1/2}\left(\bK-\bK_\ell\right)\bK_{\ell,\lambda}^{-1/2}\bK_{\ell,\lambda}^{-1/2}\bK_{m,\ell}\right|\\
    &+\left| \Km^\top \bK_{\lambda}^{-1/2}\bK_{\lambda}^{-1/2}\left(\bPsi+\sigma^2_\eps\Id\right)\bK_{\lambda}^{-1/2}\bK_{\lambda}^{-1/2}\left(\bK_{m,\ell}-\Km\right)\right|\\
    \le ~& C'_1(1+\lambda)\left\| \left(\tilde \bX^\top \tilde \bX\right)^{\odot \ell+1}-\Id\right\|_F,
\end{align}for some constant $C_1'>0$. Similarly, due to Lemma~\ref{lemm:KPsiK_vector}, \eqref{eq:E_1_1}, \eqref{eq:E_1_2} and \eqref{eq:E_1_3}, we can obtain
\begin{align}
     |E_2'-\bar E_2|\le ~& \E\left| (\bK_{m,\ell}-\Km)^\top \bK_{\ell,\lambda}^{-1/2}\bK_{\ell,\lambda}^{-1/2}\bu\right|\\
     &+ \E\left|  \Km^\top \bK_{\ell,\lambda}^{-1/2}\bK_{\ell,\lambda}^{-1/2}\left(\bK-\bK_\ell\right)\bK_{\lambda}^{-1/2}\bK_{\lambda}^{-1/2}\bu\right|\\
     \le~ &C'_2(1+\lambda)\left\| \left(\tilde \bX^\top \tilde \bX\right)^{\odot \ell+1}-\Id\right\|_F,
\end{align}for some constant $C_2'>0$. This completes the proof of \eqref{eq:Kl_test}.

\subsection{Proof of Theorem~\ref{thm:lower_bound}}
First, we state a more generic statement of the lower bound of the generalization error for RFRR. Instead of proving Theorem~\ref{thm:lower_bound}, we prove the following theorem in this section.
\begin{theorem} 
Under the assumptions of Theorem~\ref{thm:test_diff}, when $N/\log^{2C_{\sigma}}(N)\geq C_1(1+\lambda^2) n$  and $n\geq \max\{n_0,n_1\}$,  with  probability at least $1-\log^{-C_\sigma}(N)$, 
\begin{align}
  \mathcal L(\hat{f}_{\lambda}^{(\RF)})
  \ge ~&\|P_{> \ell}f^*\|^2_{2}-C_2(1+\lambda)\log^{C_{\sigma}}(N)\sqrt{\frac{n}{N}}-C_2\sqrt{n}\left\| (\bX^\top \bx)^{\odot (\ell+1)}\right\|_2,\label{eq:lower}
\end{align} 
and 
\begin{align}
 \mathcal L(\hat{f}_{\lambda}^{(\RF)})
  \ge ~ &\|P_{> \ell}f^*\|^2_{2}+\sigma_{\beps}^2\E_\bx\left[\bK_{m,\ell}^\top\bK_{\lambda,\ell}^{-2}\bK_{m,\ell}\right]-C_2\sqrt{n}\left\| (\bX^\top \bx)^{\odot (\ell+1)}\right\|_2 \\
  &-C_2(1+\lambda)\left(\left\| \left(\tilde \bX^\top \tilde \bX\right)^{\odot \ell+1}-\Id\right\|_F+\log^{C_\sigma}(N)\sqrt{\frac{n}{N}}\right), \label{eq:lower1}
\end{align}
where  $C_1$ depends only on $\sigma$, and  $C_2>0$ depends only on $\sigma,\tau$ and $\sigma_{\eps}$.
In particular, when $N/\log^{2C_{\sigma}} N\gg n,$ with high probability, 
\begin{align}
 \mathcal L(\hat{f}_{\lambda}^{(\RF)})&\ge \|P_{> \ell}f^*\|^2_{2}+\sigma_{\beps}^2\E_\bx\left[\bK_{m,\ell}^\top\bK_{\lambda,\ell}^{-2}\bK_{m,\ell}\right]-o_n(1)\\
 &\geq \|P_{> \ell}f^*\|^2_{2}-o_n(1).
\end{align}
\end{theorem}
\begin{proof}

Since $\tau \in L^2(\mathbb R,\Gamma)$, we have the following Hermite expansion:
$\tau(x)=\sum_{k=0}^{\infty} \zeta_k(\tau) h_k(x). \notag
$
Then
\begin{align*}
f^*(\bx) &= \tau(\bbeta^\top \bx)=\sum_{k=0}^{\infty} \zeta_k(\tau) h_k(\bbeta^\top \bx),\\
\left( P_{\leq \ell}f^*\right)(\bx)&=\sum_{k<\ell} \zeta_k(\tau)h_k(\bbeta^\top\bx), \quad \left( P_{> \ell}f^*\right)(\bx)=\sum_{k\geq \ell+1} \zeta_k(\tau)h_k(\bbeta^\top\bx).
\end{align*}
Similarly, we define
\begin{align}
  f^*(\bX) &= \tau(\bbeta^\top \bX)=\sum_{k=0}^{\infty} \zeta_k(\tau) h_k(\bbeta^\top \bX) \in \mathbb R^n, \\
  \left( P_{\leq \ell}f^*\right)(\bX)&=\sum_{k\leq \ell} \zeta_k(\tau)h_k(\bbeta^\top\bX), \quad \left( P_{> \ell}f^*\right)(\bX)=\sum_{k\geq \ell+1} \zeta_k(\tau)h_k(\bbeta^\top\bX), \label{eq:Pl_decompose}
\end{align}
By the property of Hermite polynomials in \eqref{eq:orthogonal_relation}, we know
\begin{align}
\E_{\bbeta} [h_j(\bbeta^\top \bx)h_k(\bbeta^\top \bx_i)]=
\delta_{jk}\langle \bx,\bx_i\rangle ^k. \notag
\end{align}
This implies
\begin{align}
    \|f^*\|_{2}^2&=\E_{\bbeta}[f^*(\bx)^2]=\sum_{k=0}^{\infty} \zeta_k(\tau)^2=\|\tau\|_{2}^2, \\
    \|P_{\leq  \ell}f^*\|_{2}^2 &=\sum_{k=0}^{\ell} \zeta^2_k(\tau), \quad\quad  \|P_{> \ell}f^*\|_{2}^2 =\sum_{k=\ell+1}^{\infty} \zeta^2_k(\tau), \\
  \E [P_{\leq \ell} f^*(\bx) P_{>\ell} f^*(\bx)]&=0.  \label{eq:orthognal_proj}
    \end{align}
From \eqref{eq:Kridge}, the predictor of the KRR is given by
\begin{align}
\hat{f}_{\lambda}^{(\KK)}(\x):=\bK(\x,\bX) (\bK(\bX,\bX)+\lambda\Id)^{-1}  \left( f^*(\bX) +\beps\right), \notag 
\end{align}
where 
\[
\bK(\x, \bX)=\sum_{k=0}^{\infty} \zeta_k^2(\sigma)  (\bx^\top \bX)^{\odot k} \in \mathbb R^{1\times n},
\]
and from Assumption \ref{assump:testdata},
\begin{align}\label{eq:normKlx}
\norm{\bK(\x, \bX)}\leq \sqrt{2n}\norm{\sigma}_4^2.
\end{align}
Define
\begin{align*}
    P_{\leq\ell}\hat{f}_{\lambda}^{(\KK)}(\x)&:=\bK(\x,\bX) (\bK(\bX,\bX)+\lambda\Id)^{-1}  \left(P_{\leq\ell} f^*(\bX) \right),\\
     P_{> \ell}\hat{f}_{\lambda}^{(\KK)}(\x)&:=\bK(\x,\bX) (\bK(\bX,\bX)+\lambda\Id)^{-1}  \left( P_{> \ell} f^*(\bX)+\beps\right) .
\end{align*}
From the orthogonal relation in \eqref{eq:orthogonal_relation} and \eqref{eq:Pl_decompose}, 
\begin{align}
\E_{\bbeta,\beps} [P_{\leq\ell}\hat{f}_{\lambda}^{(\KK)}(\x) P_{> \ell}\hat{f}_{\lambda}^{(\KK)}(\x)] &=\mathbf 0, \label{eq:orthogonal_Plfk}\\
    \E_{\bbeta}[ \left( P_{> \ell}f^*\right)(\bX)(P_{> \ell}f^*)(\bx)]&=\sum_{k\geq \ell+1 } \zeta_k^2(\tau) ((\bx^\top \bx_1)^k,\dots, (\bx^\top \bx_n)^k ).
\end{align}
Then by the linearity of expectation, we have 
\begin{align}
\E_{\bbeta,\beps}[\hat{f}_{\lambda}^{(\KK)}(\x)P_{> \ell }f^*(\x)]=\sum_{k>\ell} \zeta_k^2(\tau)\bK(\x,\bX) (\bK(\bX,\bX)+\lambda\Id)^{-1} ((\bx^\top \bx_1)^k,\dots, (\bx^\top \bx_n)^k )^\top, \notag 
\end{align}
which implies 
\begin{align}
   \left| \E_{\bbeta,\beps}[P_{> \ell }\hat{f}_{\lambda}^{(\KK)}(\x)P_{> \ell }f^*(\x)]\right|
    \leq~& \norm{\bK(\x,\bX)}\lambda_0^{-1}\sum_{k=\ell+1}^\infty\zeta_k^2(\tau) \left\| (\bX^\top \bx)^{\odot k}\right\|_2\\
    \leq ~&  \sqrt{2n}\norm{\sigma}_4^2\lambda_0^{-1} \|\tau\|_4^{2} \left(\sum_{k=\ell+1}^{\infty}\left\| (\bX^\top \bx)^{\odot k}\right\|_2^2\right)^{1/2}\\
    \leq ~& 2\sqrt{2n}\norm{\sigma}_4^2\lambda_0^{-1} \|\tau\|_4^{2} \left\| (\bX^\top \bx)^{\odot (\ell+1)}\right\|_2,\label{eq:>ell_term}
\end{align}
where the second inequality is due to    \eqref{eq:normKlx}, and the third inequality comes from  Cauchy's inequality. 
Recall the  generalization error of any predictor defined in \eqref{eq:def_testerror}. We have
\begin{align}
  \mathcal L(\hat{f}_{\lambda}^{(\KK)}) = ~&\E\left(f^*(\bx)-\hat{f}_{\lambda}^{(\KK)}(\x)\right)^2  \\
  = ~&\E\left(P_{\leq\ell }f^*(\x)+P_{>\ell }f^*(\x)-P_{\leq\ell}\hat{f}_{\lambda}^{(\KK)}(\x)-P_{> \ell}\hat{f}_{\lambda}^{(\KK)}(\x)\right)^2  \\
  =~& \E\left(P_{\leq\ell }f^*(\x)-P_{\leq\ell}\hat{f}_{\lambda}^{(\KK)}(\x)\right)^2 +\E\left(P_{>\ell }f^*(\x)-P_{>\ell}\hat{f}_{\lambda}^{(\KK)}(\x)\right)^2\\
  &+2\E\left[\left(P_{\leq\ell }f^*(\x)-P_{\leq\ell}\hat{f}_{\lambda}^{(\KK)}(\x)\right)\left(P_{>\ell }f^*(\x)-P_{>\ell}\hat{f}_{\lambda}^{(\KK)}(\x)\right)\right]\\
   \geq ~&   \E\left(P_{>\ell }f^*(\x)-P_{> \ell}\hat{f}_{\lambda}^{(\KK)}(\bx) \right)^2   \\
   = ~&\|P_{> \ell}f^*\|^2_{2}+ \E[P_{>\ell}\hat{f}_{\lambda}^{(\KK)}(\x)^2]  -2\E[P_{> \ell }f^*(\x)P_{> \ell }\hat{f}_{\lambda}^{(\KK)}(\bx)]\\
   \geq~ & \|P_{> \ell}f^*\|^2_{2}+ \E[P_{>\ell}\hat{f}_{\lambda}^{(\KK)}(\x)^2]  -4\sqrt{2n}\lambda_0^{-1}\norm{\sigma}_4^2 \|\tau\|_4^{2} \left\| (\bX^\top \bx)^{\odot (\ell+1)}\right\|_2,\label{eq:krr_lower_bound}
\end{align}
where the first inequality is due to the orthogonal relations \eqref{eq:orthognal_proj} and \eqref{eq:orthogonal_Plfk}, and 
 the second inequality is due to \eqref{eq:>ell_term}. Let $\bv=\bK_{\lambda}^{-1}\bK(\bX,\bx)$. The second term in \eqref{eq:krr_lower_bound} can be written as
\begin{align}
  \E[P_{> \ell}\hat{f}_{\lambda}^{(\KK)}(\x)^2]&=\E[(P_{> \ell}f^*(\bX)+\beps)^\top (\bK_{\lambda}^{-1} \bK(\bX,\bx)\bK(\x,\bX) \bK_{\lambda}^{-1})  \left( P_{> \ell}f^*(\bX) +\beps\right)]\\
    &=\E_{\bx}\sum_{ij}\bv_i\bv_j  \left(\E_{\bbeta} [P_{> \ell} f^*(\bx_i)P_{> \ell}f^*(\bx_j)]+\delta_{ij}\sigma_{\beps}^2\right)\\
    &=\sigma_{\beps}^2\E_{\bx}\norm{\bv}^2+ \E[P_{> \ell}f^*(\bX)^\top (\bv\bv^\top) P_{> \ell}f^*(\bX)], \label{eq:variance_term}\\
    &\geq \sigma_{\beps}^2\E_{\bx}\norm{\bv}^2=\sigma_{\beps}^2\Tr\bK_{\lambda}^{-1}\E_{\bx}[\bK(\bX,\bx)\bK(\bx,\bX)] \bK_{\lambda}^{-1}.
\end{align}

 On the other hand, from the generalization error approximation bounds in \eqref{eq:test_diff}, we obtain with   probability at least $1-\log^{-1}(N)$, when $N/\log^{2C_{\sigma}}(N)\geq C_1(1+\lambda^2)n$,
\begin{align}
\mathcal L(\hat{f}_{\lambda}^{(\RF)})&\ge\|P_{> \ell}f^*\|^2_{2}+\sigma_{\beps}^2\Tr\bK_{\lambda}^{-1}\E_{\bx}[\bK(\bX,\bx)\bK(\bx,\bX)] \bK_{\lambda}^{-1} \label{eq:lowerbound_with_variance}\\
 &\quad -C_2(1+\lambda)\log^{C_{\sigma}}(N)\sqrt{\frac{n}{N}}-C_3\sqrt{n}\left\| (\bX^\top \bx)^{\odot (\ell+1)}\right\|_2\\ 
 &\geq \|P_{> \ell}f^*\|^2_{2}-C_2(1+\lambda)\log^{C_{\sigma}}(N)\sqrt{\frac{n}{N}}-C_3\sqrt{n}\left\| (\bX^\top \bx)^{\odot (\ell+1)}\right\|_2.
\end{align}

Since we can approximate $K(\bx,\bX)$ with $K_{\ell}(\bx,\bX)$, we can apply the proof of Theorem~\ref{thm:pkrr} to obtain that
\begin{align}
   &\left| \Tr\bK_{\lambda}^{-1}\E_{\bx}[\bK(\bX,\bx)\bK(\bx,\bX)] \bK_{\lambda}^{-1}-\E_\bx\left[\bK_{m,\ell}^\top\bK_{\lambda,\ell}^{-2}\bK_{m,\ell}\right]\right|\\
   \le & \left|\E_\bx\left[(\bK_{m,\ell}-\Km)^\top\bK_{\lambda,\ell}^{-2}\bK_{m,\ell}\right]\right|+ \left|\E_\bx\left[\Km^\top\left(\bK_{\lambda,\ell}^{-1}-\Kl^{-1}\right)\bK_{\lambda,\ell}^{-1}\bK_{m,\ell}\right]\right|\\
   &+\left|\E_\bx\left[\Km^\top\bK_{m}^{-2}\left(\bK_{m,\ell}-\Km\right)\right]\right|+ \left|\E_\bx\left[\Km^\top\Kl^{-1}\left(\bK_{\lambda,\ell}^{-1}-\Kl^{-1}\right)\bK_{m,\ell}\right]\right|\\
   \le &  C(1+\lambda)\left\| \left(\tilde \bX^\top \tilde \bX\right)^{\odot \ell+1}-\Id\right\|_F
\end{align}for some constant $C>0$ depending on $\sigma,\tau,\sigma_{\eps}$, when in the last inequality, we exploit Proposition~\ref{prop:pkrr} and \eqref{eq:E_1_1}. Thus, we conclude that under the same assumptions of Theorem~\ref{thm:test_diff}, with probability at least $1-\log^{-1} N$,
\begin{align}
    \mathcal L(\hat{f}_{\lambda}^{(\RF)})\ge &\|P_{> \ell}f^*\|^2_{2}+\sigma_{\beps}^2\E_\bx\left[\bK_{m,\ell}^\top\bK_{\lambda,\ell}^{-2}\bK_{m,\ell}\right]\\
     &- C(1+\lambda)\left\| \left(\tilde \bX^\top \tilde \bX\right)^{\odot \ell+1}-\Id\right\|_F-C_2(1+\lambda)\log^{C_{\sigma}}(N)\sqrt{\frac{n}{N}}-C_3\sqrt{n}\left\| (\bX^\top \bx)^{\odot (\ell+1)}\right\|_2.
\end{align}\\
This completes the proof of the lower bound of \eqref{eq:lower1}.
\end{proof}

\end{document}